\documentclass[12pt]{amsart}
\usepackage[all]{xy}
\usepackage{graphicx}
\usepackage{tikz}  

\usepackage{amsmath}
\usepackage{amssymb}
\usepackage{amsfonts}
\usepackage{mathrsfs}
\usepackage{mathtools}

\newcommand{\Cdb}{\ensuremath{\mathbb{C}}}
\newcommand{\Ddb}{\ensuremath{\mathbb{D}}}

\newcommand{\Ndb}{\ensuremath{\mathbb{N}}}

\newcommand{\Pdb}{\ensuremath{\mathbb{P}}}

\newcommand{\Zdb}{\ensuremath{\mathbb{Z}}}

\renewcommand{\P}{\mbox{${\mathcal P}$}}

\newcommand{\norm}[1]{\Vert#1\Vert}

\newcommand{\Biggnorm}[1]{\Biggl\Vert#1\Biggr\Vert}

\newtheorem{theorem}{Theorem}[section]
\newtheorem{lemma}[theorem]{Lemma}
\newtheorem{corollary}[theorem]{Corollary}
\newtheorem{proposition}[theorem]{Proposition}
\newtheorem{definition}[theorem]{Definition}
\theoremstyle{remark}
\newtheorem{remark}[theorem]{\bf Remark}
\theoremstyle{definition}

\numberwithin{equation}{section}

\begin{document}

\title[$H^\infty$-functional calculus for commuting families]
{$H^\infty$-functional calculus for commuting families of Ritt operators and sectorial operators}

\author[O. Arrigoni]{Olivier Arrigoni}
\email{olivier.arrigoni@univ-fcomte.fr}
\author[C. Le Merdy]{Christian Le Merdy}
\email{clemerdy@univ-fcomte.fr}
\address{Laboratoire de Math\'ematiques de Besan\c con, UMR 6623, 
CNRS, Universit\'e Bourgogne Franche-Comt\'e,
25030 Besan\c{c}on Cedex, FRANCE}

\date{\today}

\maketitle

\begin{abstract}
We introduce and investigate $H^\infty$-functional calculus for 
commuting finite families of Ritt operators on Banach space $X$. 
We show that if either $X$ is a Banach lattice or $X$ or $X^*$
has property $(\alpha)$, then a commuting $d$-tuple
$(T_1,\ldots, T_d)$ of Ritt operators on $X$ has an
$H^\infty$ joint functional calculus if and only if 
each $T_k$ admits an $H^\infty$ functional calculus. Next for
$p\in(1,\infty)$,
we characterize commuting $d$-tuple
of Ritt operators on $L^p(\Omega)$
which admit an $H^\infty$ joint functional calculus, by a 
joint dilation property. 
We also obtain a 
similar characterisation for operators acting on a UMD Banach
space with property $(\alpha)$. Then we study commuting $d$-tuples
$(T_1,\ldots, T_d)$ of Ritt operators on Hilbert space. In 
particular we show that if $\norm{T_k}\leq 1$ for every $k=1,\ldots,d$, then
$(T_1,\ldots, T_d)$ satisfies a multivariable analogue of von Neumann's
inequality. Further we show analogues of most of the above results 
for commuting finite families of sectorial operators. 
\end{abstract}


\vskip 1cm

\section{Introduction}
$H^\infty$-functional calculus of
Ritt operators on Banach spaces has received a lot 
of attention recently, in connection with
discrete square functions, maximal inequalities
for discrete semigroups and ergodic theory. 
See in particular \cite{AFLM, AL, CCL, HH, LM, LMX, LMX2} 
and the references therein. This topic is closely related to 
$H^\infty$-functional calculus of sectorial operators, which itself
is fundamental for the study of harmonic analysis of semigroups and 
regularity of evolution problems. Many 
functional calculus results on sectorial operators turn out to
have discrete versions for Ritt operators, however with different 
fields of applications. We refer the reader to \cite{HH, Haa, KuW} for general
informations on $H^\infty$-functional calculus of sectorial operators.

Te main purpose of this paper is to investigate $H^\infty$-functional calculus for 
commuting finite families of Ritt operators. On the one hand, this naturally
relates to the longstanding studied polynomial functional calculus associated
to a commuting family of Hilbert space contractions and to extensions
of von Neumann's inequality. On the other hand, this is a natural 
discrete analogue of $H^\infty$-functional calculus
for commuting finite families of sectorial operators
considered in \cite{Al} and \cite{FMI} (see also \cite{LLLM} and \cite{KW1}).

For any $\gamma\in (0,\frac{\pi}{2})$, let $B_\gamma$ denote the Stolz domain
of angle $\gamma$. Given
a $d$-tuple $(T_1,\ldots,T_d)$ of commuting Ritt operators on some
Banach space $X$, we say that it admits an 
$H^\infty(B_{\gamma_1} \times \cdots \times B_{\gamma_d})$ 
joint functional calculus if it satisfies an estimate
$$
\left\| f(T_1,\ldots,T_d) \right\| \leq K \left\| f \right\|_{\infty, 
B_{\gamma_1} \times \cdots \times B_{\gamma_d}}
$$
for a large class of bounded holomorphic functions $f\colon
B_{\gamma_1} \times \cdots \times B_{\gamma_d}\to\Cdb$. Here the notation
$\left\| f \right\|_{\infty, 
B_{\gamma_1} \times \cdots \times B_{\gamma_d}}$ stands  for the
supremum norm of $f$ on $B_{\gamma_1} \times \cdots \times B_{\gamma_d}$.
See Section \ref{FC} for precise definitions and basic properties 
of functional calculus associated with $(T_1,\ldots,T_d)$. 
These extend the definitions and properties established in \cite{LM}
for a single Ritt operator.

Let us now present the main results of this paper. 
In Section \ref{Automatic} we prove the following.

\begin{theorem} \label{Th11}
Let $X$ be a Banach space. Assume that either $X$ is a Banach lattice,
or $X$ or $X^*$ has property $(\alpha)$.
Let $(T_1,\ldots,T_d)$ be a commuting $d$-tuple of Ritt operators on $X$ and assume that 
for some $0<\gamma_1,\ldots,\gamma_d<\frac{\pi}{2}$, $T_k$ has an $H^\infty(B_{\gamma_k})$ functional
calculus for any $k=1,\ldots,d$. Then for any $\gamma_k'\in(\gamma_k,\frac{\pi}{2}),\, 
k=1,\ldots,d$, $(T_1,\ldots,T_d)$ admits an 
$H^\infty(B_{\gamma_1'}\times\cdots\times B_{\gamma_d'})$ joint functional
calculus.
\end{theorem}

Note that this property does not hold true on general Banach spaces.

In Section \ref{Dilations} we characterize $H^\infty$ joint functional calculus on 
$L^p$-spaces, for $p \in (1,\infty)$), as follows.

\begin{theorem} \label{Th12}
Let $\Sigma$ be a measure space and 
let $p \in (1,\infty)$. Let $T_1,\ldots,T_d$ be commuting Ritt operators 
on $L^p(\Sigma)$. Then the $d$-tuple  $(T_1,\ldots,T_d)$ admits an 
$H^\infty(B_{\gamma_1} \times \cdots \times B_{\gamma_d})$ 
joint functional calculus for some $\gamma_k\in(0,\frac{\pi}{2})$, 
$k=1,\ldots,d$,
if and only if 
there exist a measure space $\Omega$, commuting positive 
contractive Ritt operators $R_1,\ldots,R_d$ on $L^p(\Omega)$, 
and two bounded operators $J \colon 
L^p(\Sigma) \to L^p(\Omega)$ and $Q \colon L^p(\Omega) \to L^p(\Sigma)$
such that   
\begin{equation*}
T_1^{n_1} \cdots T_d^{n_d} = Q R_1^{n_1} \cdots R_d^{n_d} J, 
\qquad (n_1,\ldots,n_d) \in \mathbb{N}^d.
\end{equation*}
\end{theorem}
 
The case $d=1$ was proved in \cite[Theorem 5.2]{AFLM}. 
The extension to $d$-tuples relies on the construction in \cite{AFLM} and
a new approach allowing to combine dilations associated to single operators
to obtain a dilation associated to a $d$-tuple. Section \ref{Dilations} also
includes a variant of Theorem \ref{Th12} for $d$-tuples of commuting Ritt operators
acting on a UMD Banach space with property $(\alpha)$.

Section \ref{Hilbert} is devoted to operators acting on Hilbert space.
It was shown in \cite{LM} that if $H$ is a Hilbert space and 
$T\colon H\to H$ is a Ritt operator, then it admits an
$H^\infty(B_\gamma)$ functional calculus for some $\gamma\in(0,\frac{\pi}{2})$
if and only if it is 
similar to a contraction, that is, there exists an invertible
$S\colon H\to H$ such that $\norm{S^{-1}TS}\leq 1$. Here we show that
if $(T_1,\ldots, T_d)$ is a commuting $d$-tuple of Ritt operators on
$H$, then $(T_1,\ldots, T_d)$ admits an $H^\infty(B_{\gamma_1}\times\cdots\times
B_{\gamma_d})$ joint
functional calculus for some
$\gamma_1,\ldots,\gamma_d \in(0,\frac{\pi}{2})$
if and only if $T_1,\ldots, T_d$ are jointly similar
to contractions, 
that is, there exists a common invertible 
$S\colon H\to H$ such that $\norm{S^{-1}T_jS}\leq 1$ for any $j=1,\ldots,d$.
We also establish the following estimate.

\begin{theorem} \label{Th13}
Let $d \geq 3$ be an integer and let $H$ be a Hilbert space. 
Let $T_1,\ldots,T_d$ be commuting contrations on $H$. Assume that
for every $j$ in $\left\lbrace 1,\ldots,d-2\right\rbrace$, 
$T_j$ is a Ritt operator.
Then there exists a constant $C \geq 1$ such that for any 
polynomial $\phi$ in $d$ variables,
\begin{equation} \label{dvn}
\left\| \phi(T_1,\ldots,T_d) \right\| 
\leq C \left\| \phi \right\|_{\infty, \mathbb{D}^d}.
\end{equation}
\end{theorem}

Note that without any Ritt type assumptions, the question whether 
any commuting $d$-tuple $(T_1,\ldots,T_d)$ of contractions on Hilbert space satisfies
an estimate (\ref{dvn}) is an open problem. See e.g. \cite[Chapter 1]{P}
for more about this.

In \cite{FMI}, E. Franks and A. McIntosh established a fundamental decomposition
of bounded holomorphic functions defined on (products of) sectors(s),
which is now known as the ``Franks-McIntosh decomposition".
Many results in Sections 3-5 heavily rely of an analogue of
this decomposition for bounded holomorphic functions defined on 
products of Stolz domains. Such a decomposition can be regarded 
as a consequence of 
\cite[Section 4]{FMI}. However the proofs in this section
of \cite{FMI} are very sketchy and the case of Stolz domains
is much simpler than the general case considered in \cite{FMI}.
For the sake of completeness we provide an ad-hoc proof in 
Section \ref{Appendix}.

In parallel to commuting families of Ritt operators,
we treat commuting families of sectorial operators. In Section \ref{FC} we give a general
definition of $H^\infty$ joint functional calculus for a $d$-tuple of
commuting sectorial operators which refines \cite{Al}. In Section \ref{Automatic}, 
we give a sectorial analogue of Theorem \ref{Th11}. In the case when $d=2$, this
result goes back to \cite{LLLM}. Section \ref{Dilations}  includes a characterisation
of $H^\infty$ joint functional calculus 
in terms of dilations, either on $L^p$-spaces 
or on UMD Banach spaces with property $(\alpha)$.

\bigskip
We end this section by fixing some notations. Throughout 
we let $B(X)$ denote the Banach algebra of 
all bounded operators on some Banach space $X$. We let $I_X$ 
denote the identity operator on $X$. 
For any (possibly unbounded) operator $A$ on $X$, we let 
$\sigma(A)$ denote the spectrum of $A$ and for every $\lambda$ in $\mathbb{C} 
\setminus \sigma(A)$, we let  $R(\lambda,A)=(\lambda I_X - A)^{-1}$ denote the
resolvent operator. Next, we let $\text{Ker}(A)$ and 
$\text{Ran}(A)$ denote the kernel and the range of $A$, respectively.

For any $a \in \mathbb{C}$ and $r>0$, $D(a,r)$ will denote the open disc 
centered at $a$ with radius $r$. Then we let $\mathbb{D} = D(0,1)$ denote
the unit disc of $\mathbb{C}$ and we set $\mathbb{T} = \overline{\mathbb{D}} 
\setminus \mathbb{D}$.

If $\mathcal{O}$ is an open non empty subset of $\mathbb{C}^d$, for some integer 
$d \geq 1$, we will denote by $H^\infty(\mathcal{O})$ the algebra of all bounded 
holomorphic functions $f \colon \mathcal{O} \to \mathbb{C}$, which is a 
Banach algebra for the norm 
\begin{equation*}
\left\| f \right\|_{\infty,\mathcal{O}} = \text{sup} \left\lbrace \left| 
f(z_1,\ldots,z_d) \right| : (z_1,\ldots,z_d) \in \mathcal{O} \right\rbrace
\end{equation*}  

If $X$ is a Banach space, $(\Omega,\mu)$ is a measure space and $p \in(1,\infty)$, 
we denote by $L_p(\Omega;X)$ the Bochner space of all measurable functions 
$f \colon \Omega \to X$ such that $\int_\Omega \left\| 
f(\omega) \right\|^p d\mu(\omega) < \infty$, and we let 
$L_p(\Omega) = L_p(\Omega;\mathbb{C}) $. We refer the reader e.g. to \cite{Hyt1} 
for more details.

The set of nonnegative integers will be denoted by $\mathbb{N} = 
\left\lbrace 0,1,2,... \right\rbrace$. We set $\Ndb^*=\Ndb\setminus\{0\}$.

In certain proofs, we use the notation $\lesssim$ to indicate an 
inequality valid up to a constant which does not depend on the particular 
elements to which it applies.

\section{Functional calculus and its basic properties}\label{FC}
We first introduce $H^\infty$-functional calculus for a commuting family of sectorial operators.
The construction and properties for a single operator go back to \cite{MI, CDMY} (see also \cite{Haa, KuW}).
The following construction is an extension (or a variant) of those in \cite{Al}
or \cite{LLLM}.

Throughout we let $X$ be an arbitrary Banach space. 
For any $\theta \in (0,\pi)$, we let
\begin{equation*}
\Sigma_\theta = \left\lbrace z \in \mathbb{C}^* : \left|\text{Arg}(z) \right| < \theta \right\rbrace.
\end{equation*} 
We say that a closed linear operator $A \colon D(A) \to X$ with dense domain $D(A) \subset X$ is 
sectorial of type $\omega \in (0,\pi)$ if $\sigma(A) \subset \overline{\Sigma_{\omega}}$ 
and for any $\theta$ in $(\omega,\pi)$, there exists a constant $C_\theta$ such that
\begin{equation} \label{sectoriel}
\left\| zR(z,A) \right\| \leq C_\theta, \qquad z \in \mathbb{C} \setminus \overline{\Sigma_{\theta}}.
\end{equation}
It is well known that $A$ is a sectorial operator of type $\omega < \frac{\pi}{2}$ 
if and only if it is the negative generator of a  bounded analytic semigroup.

Let $d \geq 1$ be an integer and let $\theta_1,\ldots,\theta_d$ be elements of $(0,\pi)$. For any 
subset $J\subset \left\lbrace 1,\ldots,d \right\rbrace $, we denote by $H^\infty_0 \left(\prod_{i \in J} 
\Sigma_{\theta_i} \right)$ the subalgebra of $H^\infty \left( \Sigma_{\theta_1} \times \cdots \times 
\Sigma_{\theta_d} \right) $ of all holomorphic bounded functions depending only on the variables 
$(z_i)_{i\in J}$ and such that there exist positive constants $c$ and $(s_i)_{i \in J}$ verifying
\begin{equation} \label{Hinfini0}
\left| f(z_1,\ldots,z_d) \right| \leq c\,\prod_{i \in J} \dfrac{|z_i|^{s_i}}{1+|z_i|^{2s_i}}, 
\qquad (z_i)_{i\in J} \in \prod_{i \in J} \Sigma_{\theta_i}.
\end{equation}
When $J=\emptyset$, $H^\infty_0 \left(\prod_{i \in \emptyset} \Sigma_{\theta_i} \right)$ 
is the space of constant functions on $\Sigma_{\theta_1} \times \cdots \times \Sigma_{\theta_d}$.

Let $(A_1,\ldots,A_d)$ be a family of commuting sectorial operators on $X$. Here the commuting property means 
that for any $k,l$ in $\left\lbrace 1,\ldots,d \right\rbrace$, the resolvent operators $R(z_k,A_k)$ and 
$R(z_l,A_l)$ commute for any $z_k$ in $ \mathbb{C} \setminus \sigma(A_k)$ and 
$z_l$ in $ \mathbb{C} \setminus \sigma(A_l)$. Assume that for every $k=1,\ldots,d$,
$A_k$ is of type $\omega_k \in (0,\theta_k)$ and let $\nu_k \in (\omega_k,\theta_k)$.

For any $f$ in $H_0^{\infty}(\prod_{i \in J} \Sigma_{\theta_i})$ with 
$J \subset \left\lbrace 1,\ldots,d \right\rbrace$, $J \neq \emptyset$, we let
\begin{equation} \label{fAi}
f(A_1,\ldots,A_d) = \left( \dfrac{1}{2\pi i } \right)^{|J|} 
\int_{\prod_{i \in J} \partial \Sigma_{\nu_i}} f(z_1,\ldots,z_d) \prod_{i \in J} R(z_i,A_i) \prod_{i \in J} dz_i,
\end{equation}
where the boundaries $\partial \Sigma_{\nu_i}$ are oriented counterclockwise for all $i$ in $J$. 
By the commuting assumption on $(A_1,\ldots,A_d)$, the product operator 
$\prod_{i \in J} R(z_i,A_i)$ is well-defined.
Further the conditions (\ref{sectoriel}) and (\ref{Hinfini0}) ensure that this integral is absolutely 
convergent and defines an element of $B(X)$. By Cauchy's Theorem, this definition does not depend on 
the choice of the $\nu_i$'s. Moreover the linear mapping $f \mapsto f(A_1,\ldots,A_d)$ is an algebra homomorphism 
from $H_0^{\infty}(\prod_{i \in J} \Sigma_{\theta_i})$ into $B(X)$. 
The proofs of these facts are similar to the ones for 
a single operator and are omitted.

If $f \equiv a$ is a constant function on $\Sigma_{\theta_1}\times\cdots\times \Sigma_{\theta_d}$
(the case when $J=\emptyset$), then we set $f(A_1,\ldots,A_d) = a I_X$.

Next we let 
$$
H_{0,1}^\infty(\Sigma_{\theta_1} \times \cdots \times \Sigma_{\theta_d})\subset 
H^\infty(\Sigma_{\theta_1} \times \cdots \times \Sigma_{\theta_d})
$$
be the sum of all
the $H^\infty_0 \left(\prod_{i \in J} 
\Sigma_{\theta_i} \right)$, with $J\subset\{1,\ldots,d\}$. 
We claim that this sum is a direct one, so that we actually have
\begin{equation}\label{Hinfini01}
H_{0,1}^\infty(\Sigma_{\theta_1} \times \cdots \times \Sigma_{\theta_d}) = 
\bigoplus_{J \subset \left\lbrace 1,\ldots,d \right\rbrace} H^\infty_0 
\left( \prod_{i \in J} \Sigma_{\theta_i} \right).
\end{equation}
Let us prove this fact. For any $i$ in $\left\lbrace 1,\ldots,d \right\rbrace$, let $p_i$ be the operator 
defined on the space $H_{0,1}^\infty(\Sigma_{\theta_1} \times \cdots \times \Sigma_{\theta_d}) $ by
\begin{equation}\label{pi}
\bigl[p_i(f)\bigr](z_1,\ldots,z_d) = f(z_1,\ldots,z_{i-1},0,z_{i+1},\ldots,z_d),
\qquad f \in H_{0,1}^\infty(\Sigma_{\theta_1} \times \cdots \times \Sigma_{\theta_d}).
\end{equation}
In this definition, $f(z_1,\ldots,z_{i-1},0,z_{i+1},\ldots,z_d)$ stands for 
the limit, when $z\in\Sigma_{\theta_i}$ and $z\to 0$, of
$f(z_1,\ldots,z_{i-1},z,z_{i+1},\ldots,z_d)$,
provided that this limit
exists. This is the case when $f$ belongs to 
$H_{0,1}^\infty(\Sigma_{\theta_1} \times \cdots \times \Sigma_{\theta_d})$.
Note that the operators $p_i$ commute.

For any $J\subset \{1,\ldots,d\}$, we can therefore define 
\begin{equation}\label{PJ}
P_{J} = \prod_{i \in J} (I-p_i) \prod_{i \in J^c} p_i.
\end{equation} 
It is easy to check that $P_{J}(f)=f$ if $f$ belongs
to $H^\infty_0 \left(\prod_{i \in J} 
\Sigma_{\theta_i} \right)$ and $P_{J}(f)=0$ if $f$ belongs
to $H^\infty_0 \left(\prod_{i \in J'} 
\Sigma_{\theta_i} \right)$ for some $J'\not=J$. The direct sum property 
(\ref{Hinfini01}) follows at once.

Moreover, 
$$
P_J\colon H_{0,1}^\infty(\Sigma_{\theta_1} \times \cdots \times \Sigma_{\theta_d})\longrightarrow 
H_{0,1}^\infty(\Sigma_{\theta_1} \times \cdots \times \Sigma_{\theta_d})
$$
is the projection 
onto $H^\infty_0 \left(\prod_{i \in J} 
\Sigma_{\theta_i} \right)$ with kernel equal to the direct sum of the 
$H^\infty_0 \left(\prod_{i \in J'} 
\Sigma_{\theta_i} \right)$, with $J'\not=J$.

For any function $f = \sum_{J \subset \left\lbrace 1,\ldots,d \right\rbrace} f_J$ in 
$H_{0,1}^\infty(\Sigma_{\theta_1} \times \cdots \times \Sigma_{\theta_d})$, 
where each $f_J$ belongs to $H^\infty_0 \left( \prod_{i \in J} \Sigma_{\theta_i} \right) $, we
naturally set
\begin{equation}
f(A_1,\ldots,A_d) = \sum_{J \subset \left\lbrace 1,\ldots,d \right\rbrace} f_J(A_1,\ldots,A_d),
\end{equation}
the operator $f_J(A_1,\ldots,A_d)$ being defined by (\ref{fAi}). In the sequel, $f\mapsto
f(A_1,\ldots,A_d)$ is called the functional calculus mapping associated with $(A_1,\ldots,A_d)$.

We note that if $f_{J}$ is in $H^\infty_0 \left( \prod_{i \in J} \Sigma_{\theta_i} \right)$ 
and $f_{J'}$ is in $H^\infty_0 \left( \prod_{i \in J'} \Sigma_{\theta_i} \right) $, 
then $f_J f_{J'}$ is in $H^\infty_0 \left( \prod_{i \in J \cup J'} \Sigma_{\theta_i} \right)$. 
Thus $H_{0,1}^\infty(\Sigma_{\theta_1} \times \cdots \times \Sigma_{\theta_d})$ is a 
subalgebra of $H^\infty(\Sigma_{\theta_1} \times \cdots \times \Sigma_{\theta_d})$.

\begin{lemma}\label{Homo}
The functional calculus mapping $f\mapsto
f(A_1,\ldots,A_d)$ is an algebra homomorphism from 
$H_{0,1}^\infty(\Sigma_{\theta_1} \times \cdots 
\times \Sigma_{\theta_d}) $ into $B(X)$.
\end{lemma}

\begin{proof}
The linearity being obvious, it suffices to check that for any subsets 
$J,J'$ of $\left\lbrace 1,\ldots,d \right\rbrace $, for any $f_J$ in $H^\infty_0 
\left( \prod_{i \in J} \Sigma_{\theta_i} \right)$ and 
$f_{J'}$ in $H^\infty_0 \left( \prod_{i \in J'} \Sigma_{\theta_i} \right) $, we have
\begin{equation}\label{morph}
f_J(A_1,\ldots,A_d) f_{J'}(A_1,\ldots,A_d) = (f_Jf_{J'})(A_1,\ldots,A_d).
\end{equation}
We let $J_0 = J \cap J'$ and we set $J_1=J\setminus J_0$ and $J'_1=J'\setminus J_0$. 
For convenience we set, for any subset $K$ of $\left\lbrace 1,\ldots,d \right\rbrace $,
\begin{equation*}
z_K=(z_i)_{i\in K},
\quad dz_K = \prod_{i \in K} dz_i,
\quad 
R_K(z_K) = \prod_{i \in K} R(z_i,A_i) 
\quad\hbox{and}\quad
\Gamma_K = \prod_{i\in K} \partial \Sigma_{\nu_i}.
\end{equation*}
Using Fubini's theorem, we have
\begin{align*}
f_J(A_1,\ldots,A_d) & f_{J'}(A_1,\ldots,A_d) \\
= & \left( \dfrac{1}{2\pi i} \right)^{|J| + |J'|} 
\left( \int_{\Gamma_J} f_{J}(z_1,\ldots,z_d) R_J(z_J)\, dz_J \right) 
\left( \int_{\Gamma_{J'}} f_{J'}(z_1,\ldots,z_d) R_{J'}(z_{J'})\, dz_{J'} \right) \\ 
= &  \left( \dfrac{1}{2\pi i} \right)^{|J| + |J'|} \int_{\Gamma_{J_1}} 
\left( \int_{\Gamma_{J_0}} f_{J}(z_1,\ldots,z_d) R_{J_0} (z_{J_0})\, dz_{J_0}\right) 
R_{J_1}(z_{J_1})\, dz_{J_1}  \\
& \qquad \times \int_{\Gamma_{J_1'}} \left( \int_{\Gamma_{J_0}} 
f_{J'}(z_1,\ldots,z_d) R_{J_0}(z_{J_0})\, dz_{J_0} \right) 
R_{J_1'}(z_{J_1'})\, dz_{J_1'}\\
= & \left( \dfrac{1}{2\pi i} \right)^{|J| + |J'|}  
\int_{\Gamma_{J_1} \times \Gamma_{J_1'}} \left[ 
\left( \int_{\Gamma_{J_0}} f_{J}(z_1,\ldots,z_d) R_{J_0}(z_{J_0})\, dz_{J_0} \right) \right.\\
& \qquad \times \left. \left( \int_{\Gamma_{J_0}} f_{J'}(z_1,\ldots,z_d) 
R_{J_0}(z_{J_0})\, dz_{J_0} \right) \right] 
R_{J_1}(z_{J_1})\, R_{J_1'}(z_{J_1'})\, dz_{J_1}dz_{J_1'}\,.
\end{align*}
For fixed variables $z_i$, for $i \notin J_0$, the two functions 
$$
(z_i)_{i \in J_0} 
\mapsto f_J(z_1,\ldots,z_d)
\qquad\hbox{and}\qquad
(z_i)_{i \in J_0} \mapsto f_{J'}(z_1,\ldots,z_d)
$$
both belong to 
$H^\infty_0 \left( \prod_{i \in J_0} \Sigma_{\theta_i} \right)$. We noticed before that 
the functional calculus mapping is a homomorphism from
$H^\infty_0\left(\prod_{i \in J_0} \Sigma_{\theta_i} \right)$ into $B(X)$. 
Consequently,
\begin{align*}
\left(\dfrac{1}{2\pi i} \right)^{2|J_0|}  
\left( \int_{\Gamma_{J_0}} f_{J} (z_1,\ldots,z_d) R_{J_0}(z_{J_0})\, dz_{J_0} \right) & 
\left( \int_{\Gamma_{J_0}} f_{J'} (z_1,\ldots,z_d) R_{J_0}(z_{J_0})\, dz_{J_0} \right) \\
=  
\left(\dfrac{1}{2\pi i} \right)^{|J_0|} &
\int_{\Gamma_{J_0}}  f_{J}f_{J'}(z_1,\ldots,z_d) R_{J_0}(z_{J_0})\, dz_{J_0}\,.
\end{align*}
Hence the above 
computation leads to
\begin{align*}
f_J&(A_1,\ldots,A_d)  f_{J'}(A_1,\ldots,A_d) \\ 
& = \left( \dfrac{1}{2\pi i} \right)^{|J| + |J'|- |J_0|} 
\int_{\Gamma_{J_1} \times \Gamma_{J_1'}}\int_{\Gamma_{J_0}} 
 f_{J}f_{J'}(z_1,\ldots,z_d) R_{J_0}(z_{J_0})\, dz_{J_0}\ 
R_{J_1}(z_{J_1}) R_{J_1'} (z_{J'_1}) dz_{J_1}dz_{J_1'} \\
& =  \left( \dfrac{1}{2\pi i} \right)^{|J| + |J'| - |J_0|}
\int_{\Gamma_{J \cup J'}}f_{J}f_{J'}(z_1,\ldots,z_d) R_{J \cup J'}(z_{J \cup J'}) dz_{J \cup J'} \\
& = (f_J f_{J'})(A_1,\ldots,A_d),
\end{align*}
since $J \cup J'$ is the disjoint union of $J_0,J_1$ and $J_1'$. This proves (\ref{morph}).
\end{proof}

\begin{definition}
We say that $(A_1,\ldots,A_d)$ admits an $H^\infty(\Sigma_{\theta_1} \times 
\cdots \times \Sigma_{\theta_d})$ joint functional calculus if the functional calculus mapping associated with
$(A_1,\ldots,A_d)$ is bounded, that is, there exists a constant $K >0$ such that 
for every $f$ in $H_{0,1}^{\infty}(\Sigma_{\theta_1} \times \cdots\times \Sigma_{\theta_d})$,
\begin{equation*}
\left\| f(A_1,\ldots,A_d) \right\| \leq K \left\| f \right\|_{\infty, 
\Sigma_{\theta_1} \times \cdots \times \Sigma_{\theta_d}}.
\end{equation*} 
\end{definition}

Each $p_i$ from (\ref{pi}) is a contraction, hence each $P_J$ from (\ref{PJ})
is a bounded operator on $H_{0,1}^{\infty}(\Sigma_{\theta_1} 
\times \cdots \times \Sigma_{\theta_d})$. This implies that 
$(A_1,\ldots,A_d)$ admits an $H^\infty(\Sigma_{\theta_1} \times 
\cdots \times \Sigma_{\theta_d})$ joint functional calculus if and only
if $f\mapsto f(A_1,\ldots, A_d)$ is bounded on 
$H^\infty_0 \left( \prod_{i \in J} \Sigma_{\theta_i} \right)$ for any $J\subset\{1,\ldots,d\}$.
Consequently if $(A_1,\ldots,A_d)$ admits an $H^\infty(\Sigma_{\theta_1} \times 
\cdots \times \Sigma_{\theta_d})$ joint functional calculus, then 
every subfamily $(A_i)_{i \in J}$, where $J \subset \left\lbrace 1,\ldots,d \right\rbrace$,
also admits an $H^\infty( \prod_{i \in J} \Sigma_{\theta_i})$ joint functional calculus. 
In particular, for every $k=1,\ldots,d$, $A_k$ admits an $H^\infty(\Sigma_{\theta_k})$ functional calculus 
in the usual sense (see \cite[Chapter 5]{Haa}).

We now turn to Ritt operators. Recall that a bounded 
operator $T \colon X \to X$ is called a Ritt operator if there exists a constant $C>0$ such that
\begin{equation*}
\left\| T^n \right\| \leq C \qquad\hbox{and}\qquad 
\left\| n(T^n-T^{n-1}) \right\| \leq C, \qquad n \geq 1.
\end{equation*}
Ritt operators have a spectral characterisation. 
Namely $T$ is a Ritt operator if and only if $\sigma(T) \subset \overline{\mathbb{D}}$ 
and there exists a constant $K>0$ such that
\begin{equation*}
\left\| (\lambda-1)R(\lambda,T) \right\| \leq K, \qquad \lambda \in \mathbb{C},\ |\lambda| > 1.
\end{equation*}
There is a simple link between sectorial operators and Ritt operators. Indeed
if we let $A = I_X-T$, then $T$ is a Ritt operator if and only 
if $ \sigma(T) \subset \mathbb{D} \cup \left\lbrace 1 \right\rbrace$ and 
$A$ is a sectorial operator of type $\omega < \frac{\pi}{2}$. Equivalently, 
$T$ is a Ritt operator if and only if $ \sigma(T) \subset \mathbb{D} 
\cup \left\lbrace 1 \right\rbrace$ and $(e^{-t(I_X-T)})_{t \geq 0}$ is a bounded analytic semigroup.

For any  $\alpha$ in $(0,\frac{\pi}{2})$, let $B_\alpha$ denote the \textit{Stolz domain} of angle $\alpha$, 
defined as the interior of the convex hull of $1$ and the disc $ D(0,\sin(\alpha))$.

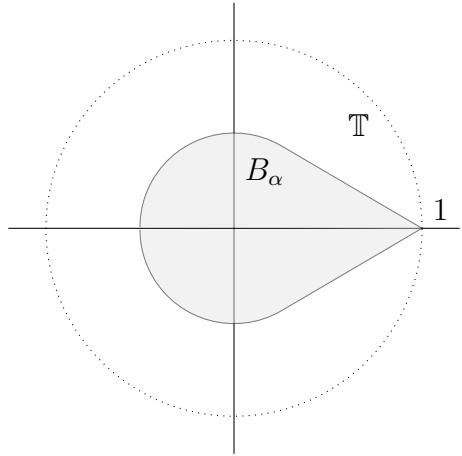
\begin{center}
\begin{tikzpicture}
[scale = 2.5];
\draw (-1.2,0)--(1.2,0);
\draw (0,-1.2)--(0,1.2);
\draw [dotted] circle(1);

\draw [fill=gray!20,opacity=0.5] (1,0)--(0.25,0.44)--(0.25,0.44) arc (60:180:0.5);
\draw [fill=gray!20,opacity=0.5] (1,0)--(0.25,-0.44)--(0.25,-0.44) arc (-60:-180:0.5);

\draw [right] (0,0.3) node{$B_\alpha$};

\draw [right] (0.55,0.55) node{$\mathbb{T}$};

\draw [right] (1,0.1) node{$1$};

\end{tikzpicture}
\end{center}

It turns out that if $T$ is a Ritt operator, then $\sigma(T) \subset \overline{B_\alpha}$ for some
$\alpha$ in $(0,\frac{\pi}{2})$.
More precisely (see \cite[Lemma 2.1]{LM}), one can find $\alpha \in (0,\frac{\pi}{2})$ 
such that $\sigma(T) \subset \overline{B_\alpha} $ and for any $\beta \in (\alpha,\frac{\pi}{2})$, 
there exists a constant $K_\beta>0$ such that
\begin{equation} \label{majRittbeta}
\left\| (\lambda-1)R(\lambda,T) \right\| \leq K_{\beta}, 
\qquad \lambda \in \mathbb{C} \setminus \overline{B_\beta}.
\end{equation}
If this property holds, then we say that $T$ is a Ritt operator of type $\alpha$. 
We refer to \cite{Ly,NZ,Nev} for the above facts and also to \cite{LM}
and the references therein for complements on the class of Ritt operators.

$H^\infty$-functional calculus for Ritt operators was formally introduced in \cite{LM}. We now extend
this definition to commuting families. We follow the same pattern as for families of sectorial operators.

Let $d \geq 1$ be an integer and 
let $\gamma_1,\ldots,\gamma_d$ be elements of $(0,\frac{\pi}{2})$. For any 
subset $J$ of $\left\lbrace 1,\ldots,d \right\rbrace $, we denote by $H^\infty_0 
\left( \prod_{i \in J} B_{\gamma_i} \right) $ the subalgebra of $H^\infty 
\left( B_{\gamma_1} \times \cdots \times B_{\gamma_d} \right) $ of all holomorphic bounded 
functions $f$ depending only on variables 
$(\lambda_i)_{i\in J}$ and such that there exist positive constants $c$ and $(s_i)_{i \in J}$ verifying
\begin{equation} \label{Hinfini0Ritt}
\left| f(\lambda_1,\ldots,\lambda_d) \right| \leq c ~ \prod_{i \in J} 
|1-\lambda_i|^{s_i}, \qquad (\lambda_i)_{i\in J} \in \prod_{i \in J} B_{\gamma_i}.
\end{equation}
When $J=\emptyset$, $H^\infty_0 \left( \prod_{i \in \emptyset} B_{\gamma_i} \right) $ is the space of constant 
functions on $B_{\gamma_1} \times \cdots \times B_{\gamma_d}$.

Let $(T_1,\ldots,T_d)$ be a $d$-tuple of commuting Ritt operators. Assume that for any $k=1,\ldots,d$,
$T_k$ is of type $\alpha_k \in (0,\gamma_k)$ and let $\beta_k \in (\alpha_k,\gamma_k)$.

For any $f$ in $H_0^{\infty}(\prod_{i \in J} B_{\gamma_i})$ with 
$J \subset \left\lbrace 1,\ldots,d \right\rbrace$, $J \neq \emptyset$, we let
\begin{equation} \label{fTi}
f(T_1,\ldots,T_d) = \left( \dfrac{1}{2\pi i} \right)^{|J|} \int_{\prod_{i \in J} \partial B_{\beta_i}} f(\lambda_1,\ldots,\lambda_d) \prod_{i \in J} R(\lambda_i,T_i) \prod_{i \in J} d \lambda_i,
\end{equation}
where the $\partial B_{\beta_i}$ are oriented counterclockwise for all $i \in J$. 
This integral is absolutely convergent,
hence defines an element of $B(X)$, its definition does not depend on the $\beta_i$ and 
the linear mapping $f \mapsto f(T_1,\ldots,T_d)$ is an algebra homomorphism 
from $H_0^{\infty}(\prod_{i \in J} B_{\gamma_i})$ into $B(X)$. If $f \equiv a$ is a constant function, 
then we let $f(T_1,\ldots,T_d) = aI_X$.

Next we define
$$
H_{0,1}^\infty(B_{\gamma_1} \times \cdots \times B_{\gamma_d}) = 
\bigoplus_{J \subset \left\lbrace 1,\ldots,d \right\rbrace} 
H^\infty_0 \left( \prod_{i \in J} B_{\gamma_i} \right).
$$
As in the sectorial case, the above sum is indeed a direct one. 
More precisely, set  
\begin{equation*}
\bigl[q_i(f)\bigr](\lambda_1,\ldots,\lambda_d) = 
f(\lambda_1,\ldots,\lambda_{i-1},1,\lambda_{i+1},\ldots,\lambda_d), 
\qquad f \in H_{0,1}^\infty(B_{\gamma_1} \times \cdots \times B_{\gamma_d}),
\end{equation*}
for $i=1,\ldots,d$, and 
\begin{equation} \label{QJ}
Q_J = \prod_{i \in J} (I - q_j)\prod_{i \in J^c} q_i,
\end{equation}
for $J \subset \left\lbrace 1,\ldots,d \right\rbrace$.
These mappings are well-defined and 
$$
Q_J\colon H_{0,1}^\infty(B_{\gamma_1} \times \cdots \times B_{\gamma_d})\longrightarrow
H_{0,1}^\infty(B_{\gamma_1} \times \cdots \times B_{\gamma_d})
$$ 
is the projection onto $H^\infty_0 \left( \prod_{i \in J} B_{\gamma_i} \right)$ 
with kernel equal to the direct sum of the spaces 
$H^\infty_0 \left( \prod_{i \in J'} B_{\gamma_i} \right)$, with $J'\not=J$.

For any function $f = \sum_{J \subset \left\lbrace 1,\ldots,d \right\rbrace} f_J$
in $H_{0,1}^\infty(B_{\gamma_1} \times \cdots \times B_{\gamma_d})$,
with $f_J\in H^\infty_0 \left( \prod_{i \in J} B_{\gamma_i} \right)$, 
we let $f(T_1,\ldots,T_d) = \sum_{J \subset \left\lbrace 1,\ldots,d \right\rbrace} f_J(T_1,\ldots,T_d)$,
where every $f_J(T_1,\ldots,T_d)$ is defined by (\ref{fTi}). The mapping
$f\mapsto f(T_1,\ldots,T_d)$ is called the functional calculus mapping associated with
$(T_1,\ldots,T_d)$. As in the sectorial case (see Lemma \ref{Homo}), one shows that 
this is an algebra homomorphism from $H_{0,1}^\infty(B_{\gamma_1} \times \cdots \times B_{\gamma_d})$ 
into $B(X)$.

\begin{definition}
We say that $(T_1,\ldots,T_d)$ admits an $H^\infty(B_{\gamma_1} \times \cdots \times B_{\gamma_d})$ 
joint functional calculus if the above functional calculus mapping is bounded,
that is, there exists a constant $K >0$ such that for every $f$ 
in $H_{0,1}^\infty(B_{\gamma_1} \times \cdots \times B_{\gamma_d})$, we have
\begin{equation} \label{calcjointRitt}
\left\| f(T_1,\ldots,T_d) \right\| \leq K \left\| f \right\|_{\infty, B_{\gamma_1} 
\times \cdots \times B_{\gamma_d}}.
\end{equation} 
\end{definition}

As in the sectorial case, we observe that $(T_1,\ldots,T_d)$
admits an $H^\infty(B_{\gamma_1} \times \cdots \times B_{\gamma_d})$ 
joint functional calculus if and only if 
$f\mapsto f(T_1,\ldots, T_d)$ is bounded on 
$H^\infty_0 \left( \prod_{i \in J} B_{\gamma_i} \right)$ for any $J\subset\{1,\ldots,d\}$.
This follows from the fact that  each $q_i$ is a contraction, hence each $Q_J$ is bounded.

Further if $(T_1,\ldots,T_d)$ admits an $H^\infty(B_{\gamma_1} \times \cdots \times 
B_{\gamma_d})$  joint functional calculus, then 
for every $k=1,\ldots,d$, $T_k$ admits an $H^\infty(\Sigma_{\theta_k})$ functional calculus 
in the  sense of \cite[Definition 2.4]{LM}.

It is natural to consider polynomial functional calculus in this context.
We let $\P_d$ denote the algebra of all complex valued polynomials 
in $d$ variables. Clearly $\P_d$ can be regarded as a subalgebra of 
$H_{0,1}^\infty(B_{\gamma_1} \times \cdots \times B_{\gamma_d})$
and for $\phi\in\P_d$, the definition of $\phi(T_1,\ldots,T_d)$ given
by replacing the variables $(z_1,\ldots,z_d)$ by the operators
$(T_1,\ldots,T_d)$ coincides with the
one given by the functional calculus mapping. This follows
from the basic properties of the Dunford-Riesz functional calculus.
We will show below that to obtain 
an  $H^\infty(B_{\gamma_1} \times\cdots \times B_{\gamma_d})$ joint functional calculus 
for $(T_1,\ldots,T_d)$, it suffices to consider polynomials in (\ref{calcjointRitt}).

To prove this result, we will use the following form of Runge's Lemma.

\begin{lemma}\label{Runge}
Let $d \geq 1$ be an integer and $V_1,\ldots,V_d$ be compact subsets of $\mathbb{C}$ such 
that $\mathbb{C} \setminus V_i$ is connected for all $i=1,\ldots,d$. 
Let $\Omega_1,\ldots,\Omega_d$ be open subsets of $\mathbb{C}$ such that $V_i \subset \Omega_i$ for 
all $i=1,\ldots,d$. 
Let $f\colon \Omega_1 \times \cdots \times \Omega_d\to \Cdb\,$
be a holomorphic function. Then there exists a sequence $(\phi_m)_{m\geq 1}$ in $\mathcal{P}_d$ 
which converges uniformly to $f$ on $V_1 \times \cdots \times V_d$.
\end{lemma}

In the case $d=1$, this statement is \cite[Theorem 13.7]{Rud}. The proof of the latter 
readily extends to the $d$-variable case so we omit it.

\begin{proposition}\label{Rittpolynome}
Let $d \geq 1$ be an integer and let $(T_1,\ldots,T_d)$ be a commuting family of 
Ritt operators. Let $\gamma_i \in (0,\frac{\pi}{2})$ for $i=1,\ldots,d$. 
The following assertions are equivalent.
\begin{itemize}
\item[(i)] $(T_1,\ldots,T_d)$ admits an $H^\infty(B_{\gamma_1} \times \cdots \times B_{\gamma_d})$ 
joint functional calculus.
\item[(ii)] There exists a constant $K >0$ such that for any $\phi \in \mathcal{P}_d$ we have
\begin{equation} \label{calcpolRitt}
\left\| \phi (T_1,\ldots,T_d) \right\| 
\leq K \left\| \phi \right\|_{\infty, B_{\gamma_1} \times \cdots\times B_{\gamma_d}}.
\end{equation}
\end{itemize}
\end{proposition}

\begin{proof}
The implication
(i) $\Rightarrow$ (ii) is obvious. Conversely assume (ii). 
As noticed after (\ref{calcjointRitt}) it suffices, to prove (i), 
to establish the boundedness of $f\mapsto f(T_1,\ldots,T_d)$ on 
$H^\infty_0\left(\prod_{i\in J} B_{\gamma_i}\right)$
for any $J\subset\{1,\ldots,d\}$.
By induction, it actually suffices to prove the estimate
\begin{equation}\label{goal}
\left\| f (T_1,\ldots,T_d) \right\| \leq K \left\| f \right\|_{\infty, 
B_{\gamma_1} \times \cdots \times B_{\gamma_d}}
\end{equation} 
for any $f$ in $H_{0}^{\infty}(B_{\gamma_1} \times \cdots \times B_{\gamma_d})$.

Let $f$ be such a function and consider
$r \in (0,1)$ and $r' \in (r,1)$. Let $\Gamma=
\partial (r' B_{\gamma_1}) \times \cdots \times \partial(r' B_{\gamma_d})$, 
where all the $ \partial(r' B_{\gamma_{i}})$ are oriented counterclockwise. 
By Lemma \ref{Runge} applied with $V_i = r' \overline{B_{\gamma_i}}$ and $\Omega_i=B_{\gamma_i}$, 
there exists a sequence $(\phi_m)_{m\geq 1}$ of $\mathcal{P}_d$ such that 
$\phi_m \rightarrow f$ uniformly on
the compact set $r' \overline{B_{\gamma_1}} \times \cdots \times  r' \overline{B_{\gamma_d}}$.

Since we have $\sigma(rT_i) \subset r' B_{\gamma_i}$ for all $i=1,\ldots,d$, 
the Dunford-Riesz functional calculus provides 
\begin{equation*}
\phi_m(rT_1,\ldots,rT_d)  = \left( \dfrac{1}{2 \pi i} \right)^d \int_\Gamma 
\phi_m(\lambda_1,\ldots,\lambda_d) R(\lambda_1,rT_1) \cdots R(\lambda_d,rT_d) \, d\lambda_1 \cdots d\lambda_d
\end{equation*}
and
\begin{equation*}
f(rT_1,\ldots,rT_d) = \left( \dfrac{1}{2 \pi i} \right)^d \int_\Gamma 
f (\lambda_1,\ldots,\lambda_d) R(\lambda_1,rT_1) \cdots R(\lambda_d,rT_d) \, d\lambda_1 \cdots d\lambda_d.
\end{equation*}
The uniform convergence of $(\phi_m)_{m\geq 1}$ to $f$ on 
$r' \overline{B_{\gamma_1}} \times \cdots \times  r' \overline{B_{\gamma_d}}$ implies that 
\begin{equation*}
\phi_m(rT_1,\ldots,rT_d) \underset{m \to \infty}{\longrightarrow} f(rT_1,\ldots,rT_d)
\quad\hbox{and}\quad
\left\| \phi_m \right\|_{\infty, r' B_{\gamma_1} 
\times \cdots \times r' B_{\gamma_d}}\underset{m \to \infty}{\longrightarrow}
\left\| f \right\|_{\infty, r' B_{\gamma_1} 
\times \cdots \times r' B_{\gamma_d}}
\end{equation*}
Using (\ref{calcpolRitt}) we have, for any interger $m\geq 1$,
\begin{equation*}
\left\| \phi_m(rT_1,\ldots,rT_d) \right\| 
\leq K \left\| \phi_m \right\|_{\infty, r B_{\gamma_1} \times 
\cdots \times r B_{\gamma_d}} \leq K\left\| \phi_m \right\|_{\infty, r' B_{\gamma_1} 
\times \cdots \times r' B_{\gamma_d}}.
\end{equation*}
Passing to the limit when $m \to \infty$, we deduce that 
\begin{equation*}
\left\| f (rT_1,\ldots,rT_d) \right\| \leq K 
\left\| f \right\|_{\infty, r' B_{\gamma_1} \times \cdots \times r' B_{\gamma_d}}.
\end{equation*}
Finally, we have $\displaystyle{\lim_{r \to 1} f(rT_1,\ldots,rT_d)} = f(T_1,\ldots,T_d)$ 
by Lebesgue's dominated convergence Theorem. We deduce (\ref{goal}).
\end{proof}

\section{Automaticity of the $H^\infty$ joint funtional calculus}\label{Automatic}
Let $(T_1,\ldots,T_d)$ be a commuting family of Ritt operators on some Banach space $X$. If 
this $d$-tuple admits an $H^\infty$ joint functional calculus, then each $T_k$ admits 
an $H^\infty$ functional calculus (see Section \ref{FC}). The purpose of this section is to show that
the converse holds true if either $X$ is a Banach lattice or $X$ (or its dual space $X^*$) has property $(\alpha)$.
A similar result is also established in the sectorial case, see Theorem \ref{Th32} below.

We refer the reader to \cite{LT2} for definitions and basic properties of Banach lattices. 

In order to define property $(\alpha)$, and also for further purposes, 
we need some background on Rademacher averages. Let $I$ be a countable set and let 
$(r_k)_{k \in I}$ be a independent family of Rademacher variables on some probability 
space $(\Omega_0,\mathbb{P})$. Let $X$ be a Banach space.
If $(x_k)_{k\in I}$ is a finitely supported family in
$X$, we let
\begin{equation*}
\left\| \sum_{k\in I} r_k \otimes x_k \right\|_{\text{Rad}(I;X)} = 
\left( \int_{\Omega_0} \left\| \sum_{k\in I} r_k(t) 
\, x_k\right\|_X^2 d\mathbb{P}(t) \right)^\frac{1}{2}.
\end{equation*}
This is the norm of $\sum_{k\in I} r_k \otimes x_k$ is $L^2(\Omega_0;X)$. 
The closure of all finite sums $\sum_{k\in I} r_k \otimes x_k$
in $L^2(\Omega_0;X)$ will be denoted by 
$\text{Rad}(I;X)$. In the case when $I=\Ndb^*$, we write 
$\text{Rad}(X) =\text{Rad}(\Ndb^*; X)$ for simplicity.

We say that $X$ has property $(\alpha)$ 
if there exists a constant $C>0$ such that for any integer $n \geq 1$, for
any family $(a_{i,j})_{1 \leq i,j \leq n}$
of complex numbers and for any family $(x_{i,j})_{1 \leq i,j \leq n}$
in $X$, we have
\begin{equation} \label{alpha}
\left\| \sum_{1 \leq i,j \leq n} a_{i,j} r_i \otimes r_j 
\otimes x_{i,j} \right\|_{\text{Rad}(\text{Rad}(X))} \leq  
C \,\sup_{i,j} \left\{\left| a_{i,j} \right| \right\} 
\left\| \sum_{1 \leq i,j \leq n} r_i \otimes r_j \otimes x_{i,j} 
\right\|_{\text{Rad}(\text{Rad}(X))}.
\end{equation}
This property was introduced by Pisier in \cite{P2}. It plays a key role in many
issues related to $H^\infty$-functional calculus, 
see in particular \cite{KW1,KW2,LLLM,LM}.

We recall that all Banach lattices with finite cotype have property
$(\alpha)$. In particular for any $p\in[1,\infty)$, $L^p$-spaces have 
property $(\alpha)$. On the contrary, infinite dimensional noncommutative $L^p$-spaces
(for $p\not=2$) do not have 
property $(\alpha)$. This goes back to \cite{P2}.

The main result of this section is the following.

\begin{theorem} \label{Th32}
Let $X$ be a Banach space. Assume that either $X$ is a Banach lattice,
or $X$ or $X^*$ has property $(\alpha)$.
Let $d\geq 2$ be an integer. Then the following two properties hold :
\begin{itemize}
\item [(\textbf{P1})] Let $(T_1,\ldots,T_d)$ be a commuting 
$d$-tuple of Ritt operators on $X$ and assume that 
for some $0<\gamma_1,\ldots,\gamma_d<\frac{\pi}{2}$,  $T_k$ has an $H^\infty(B_{\gamma_k})$ functional
calculus for any $k=1,\ldots,d$. Then for any $\gamma_k'\in(\gamma_k,\frac{\pi}{2}),\, 
k=1,\ldots,d$, $(T_1,\ldots,T_d)$ admits an $H^\infty(B_{\gamma_1'}\times\cdots\times B_{\gamma_d'})$ joint functional
calculus.
\item [(\textbf{P2})] Let $(A_1,\ldots,A_d)$ be a commuting $d$-tuple of sectorial operators 
on $X$ and assume that 
for some $0<\theta_1,\ldots,\theta_d<\pi$,  $A_k$ has an $H^\infty(\Sigma_{\theta_k})$ functional
calculus for any $k=1,\ldots,d$. Then for any $\theta_k'\in(\theta_k,\pi),\, 
k=1,\ldots,d$, $(A_1,\ldots,A_d)$ admits an 
$H^\infty(\Sigma_{\theta_1'}\times \cdots \times \Sigma_{\theta_d'})$ joint functional calculus.
\end{itemize}
\end{theorem}

Property (\textbf{P2}) for $d=2$ was proved in \cite{LLLM}.
The proof for $d\geq 3$ is a simple adaptation of the argument devised 
in the latter 
paper. In the special case when $X$ is an $L^p$-space for $p\in[1,\infty)$, 
property (\textbf{P2}) goes back to \cite{Al}.
Proving property (\textbf{P1}) will require the Franks-McIntosh decomposition 
presented in the Appendix.

To proceed we need more ingredients on Rademacher averages. 
Let $d\geq 1$ be an integer.

We denote by $\text{Rad}^d(X)$ the closure in $L^2(\Omega_0^d;X)$ 
of the space of all elements of the form
\begin{equation*}
\sum_{1 \leq i_1,\ldots,i_d \leq n} r_{i_1} \otimes \cdots \otimes r_{i_d} 
\otimes x_{i_1,\ldots,i_d}, \qquad n \in \mathbb{N}^*, \qquad x_{i_1,\ldots,i_d} \in X. 
\end{equation*}
Clearly we can rewrite this space as
\begin{equation}
\text{Rad}^d(X) = \underset{d \text{ times}}{\underbrace{\text{Rad}(\text{Rad}(\cdots\text{Rad}}}(X)
\cdots)).
\end{equation}
For convenience we set 
\begin{equation} \label{Ndnorm}
N_d\left([x_{i_1,\ldots,i_d}]\right) = 
\left\| \sum_{1 \leq i_1,\ldots,i_d \leq n} r_{i_1} \otimes \cdots \otimes 
r_{i_d} \otimes x_{i_1,\ldots,i_d} \right\|_{\text{Rad}^d(X)} 
\end{equation}
for any family $(x_{i_1,\ldots,i_d})_{1 \leq i_1,\ldots,i_d \leq n}$ in $X$.

We will say that $X$ satisfies property $(A_d)$ if there exists a constant 
$C>0$ such that for any integer $n \geq 1$, 
for any family of complex numbers $(a_{i_1,\ldots,i_d})_{1 \leq i_1,\ldots,i_d \leq n}$ 
and for any families $(x_{i_1,\ldots,i_d})_{1 \leq i_1,\ldots,i_d \leq n}$ in $X$ and 
$(x_{i_1,\ldots,i_d}^*)_{1 \leq i_1,\ldots,i_d \leq n}$ in $X^*$, we have
\begin{equation} \label{majAd}
\left| \sum_{i_1,\ldots,i_d} a_{i_1,\ldots,i_d} \langle x_{i_1,\ldots,i_d}^*, x_{i_1,\ldots,i_d} 
\rangle \right| \leq C  \sup_{i_1,\ldots,i_d} 
\left\{\left| a_{i_1,\ldots,i_d}\right|\right\}\, N_d\left([x_{i_1,\ldots,i_d}]\right)\, 
N_d\left([x_{i_1,\ldots,i_d}^*]\right).
\end{equation}

Theorem \ref{Th32} is a straightforward consequence of the 
next three propositions, that will be proved in the rest of this section.

\begin{proposition} \label{prop32}
If $X$ satisfies property $(A_d)$ for some integer $d \geq 2$, then \textbf{(P1)} and \textbf{(P2)} hold true.
\end{proposition}

\begin{proposition} \label{prop33}
Every Banach lattice satisfies property $(A_d)$ for every integer $d \geq 2$.
\end{proposition}

\begin{proposition} \label{prop34}
If $X$ or $X^*$ has property $(\alpha)$, then $X$ satisfies 
property $(A_d)$ for every integer $d \geq 2$.
\end{proposition}

\begin{proof}[Proof of Proposition \ref{prop32}] Assume that 
$X$ satisfies property $(A_d)$ for some $d\geq 2$.
We only prove 
(\textbf{P1}), the proof of (\textbf{P2}) being similar. We consider   
commuting Ritt operators $T_1,\ldots,T_d$ such that, for every $k=1,\ldots,d$, $T_k$ has a 
bounded $H^{\infty}(B_{\gamma_k})$ functional calculus. Let $\gamma_k'$ in $(\gamma_k,\frac{\pi}{2})$.
By Section \ref{FC}, and a simple induction argument, it suffices to have an estimate 
$\left\| h(T_1,\ldots,T_d) 
\right\| \lesssim \left\| h \right\|_{\infty,B_{\gamma_1'} \times \cdots \times 
B_{\gamma_d'}}$ for functions $h$ in $H^\infty_{0}(B_{\gamma'_1} \times \cdots 
\times B_{\gamma'_d})$.

For $h\in H^\infty_{0}(B_{\gamma'_1} \times \cdots \times 
B_{\gamma'_d})$, we consider the Franks-McIntosh decomposition given by Theorem 
\ref{FranksMcIntoshplusvariables}. According to 
this statement we may write, for every $(\zeta_1,\ldots,\zeta_d) 
\in \prod_{k=1}^d B_{\gamma_k}$,
\begin{equation}\label{h}
h(\zeta_1,\ldots,\zeta_d) = \sum_{(i_1,\ldots,i_d) \in \mathbb{N}^{*d}} 
a_{i_1,\ldots,i_d} \Psi_{1,i_1}(\zeta_1) \tilde{\Psi}_{1,i_1}(\zeta_1) \cdots 
\Psi_{d,i_d}(\zeta_d) \tilde{\Psi}_{d,i_d}(\zeta_d),
\end{equation}
where $(a_{i_1,\ldots,i_d})$ is a family of complex numbers
satisfying an estimate
\begin{equation} \label{aid}
\left| a_{i_1,\ldots,i_d} \right| \lesssim \left\| h \right\|_{\infty, B_{\gamma_1'} 
\times \cdots \times B_{\gamma_d'}}, \qquad (i_1,\ldots,i_d) \in \mathbb{N}^{*d},
\end{equation}
the functions $\Psi_{k,i_k}$ and  
$\tilde{\Psi}_{k,i_k}$ belong to $H_0^\infty(B_{\gamma_k})$ and they satisfy inequalities
\begin{equation} \label{Psiid}
\sup\left\lbrace \sum_{i_k = 1}^{\infty} \left| \Psi_{k,i_k}(\zeta_k)\right| \, :\, 
\zeta_k \in B_{\gamma_k} \right\rbrace \leq C \quad
\hbox{and}\quad \sup\left\lbrace \sum_{i_k = 1}^{\infty} \left| 
\tilde{\Psi}_{k,i_k}(\zeta_k)\right| : \zeta_k \in B_{\gamma_k} \right\rbrace \leq C
\end{equation}
for every $k=1,\ldots,d$, and for a constant $C>0$ not depending on $h$.

We consider the partial sums in (\ref{h}), defined for every $n\geq 1$ and every
$(\zeta_1,\ldots,\zeta_d)$ in $\prod_{k=1}^d B_{\gamma_k} $ by
\begin{equation} \label{hn}
 h_n(\zeta_1,\ldots,\zeta_d) = \sum_{1 \leq i_1,\ldots,i_d \leq n} a_{i_1,\ldots,i_d} 
\Psi_{1,i_1}(\zeta_1) \tilde{\Psi}_{1,i_1}(\zeta_1)
\cdots\Psi_{d,i_d}(\zeta_d) \tilde{\Psi}_{d,i_d}(\zeta_d).
\end{equation}
The functions $\Psi_{k,i_k}$ and $\tilde{\Psi}_{k,i_k}$ both belong
to $H_0^\infty(B_{\gamma_k})$ hence this implies
\begin{equation} \label{hn}
 h_n(T_1,\ldots,T_d) = \sum_{1 \leq i_1,\ldots,i_d \leq n} a_{i_1,\ldots,i_d} \Psi_{1,i_1}(T_1) 
\tilde{\Psi}_{1,i_1}(T_1)\cdots \Psi_{d,i_d}(T_d) \tilde{\Psi}_{d,i_d}(T_d).
\end{equation}

Let us prove the existence of a constant $K>0$, not depending either on $n$ or $h$,  
such that 
\begin{equation} \label{majhn}
\left\| h_n(T_1,\ldots,T_d) \right\| \leq K \left\| h \right\|_{\infty,B_{\gamma_1'} 
\times \cdots \times B_{\gamma_d'}}.
\end{equation}
We let $x \in X$ and $x^* \in X^*$. Applying (\ref{hn}), we write
\begin{align*}
\langle x^*, h_n(T_1,\ldots,T_d)x \rangle\,=\, \sum_{1 
\leq i_1,\ldots,i_d \leq n} a_{i_1,\ldots,i_d} \bigl\langle \tilde{\Psi}_{1,i_1}(T_1)^*\cdots
\tilde{\Psi}_{d,i_d}(T_d)^* x^*, \Psi_{1,i_1}(T_1)\cdots\Psi_{d,i_d}(T_d) x \bigr\rangle.
\end{align*}
We let 
$$
x_{i_1,\ldots, i_d} = \Psi_{1,i_1}(T_1)\cdots\Psi_{d,i_d}(T_d) x
\quad\hbox{and}\quad
x_{i_1,\ldots,i_d}^* = 
\tilde{\Psi}_{1,i_1}(T_1)^*\cdots\tilde{\Psi}_{d,i_d}(T_d)^* x^*.
$$ 
Using property $(A_d)$ and the estimate (\ref{aid}), we have
\begin{equation*}
\left| \langle x^*, h_n(T_1,\ldots,T_d)x \rangle \right| 
\lesssim \left\| h \right\|_{\infty,B_{\gamma_1'} \times 
\cdots \times B_{\gamma_d'}} N_d\left([x_{i_1,\ldots, i_d}]\right) N_d\left([x_{i_1,\ldots, i_d}^*]\right)
\end{equation*}
Let us momentarilty fix some $(t_1,\ldots,t_d)$ in $\Omega_0^d$. By (\ref{Psiid}) and 
the $H^\infty(B_{\gamma_k})$ functional calculus property of $T_k$ for all $k=1,\ldots,d$,
we have estimates
\begin{align*}
\left\| \sum_{1 \leq i_1,\ldots,i_d \leq n} r_{i_1}(t_1)
\cdots r_{i_d}(t_d) \Psi_{1,i_1}(T_1)\cdots \Psi_{d,i_d}(T_d) x \right\|
& \leq \prod_{k=1}^d \left( \left\| \sum_{i_k=1}^n r_{i_k}(t_k) 
\Psi_{k,i_k}(T_k) \right\| \right) \left\| x \right\|\\
& \lesssim \prod_{k=1}^d \left( \left\| \sum_{i_k=1}^n r_{i_k}(t_k) 
\Psi_{k,i_k} \right\|_{\infty,B_{\gamma_k}} \right) \left\| x \right\|\\
&  \lesssim \prod_{k=1}^d \left( \left\| \sum_{i_k=1}^n \left| 
\Psi_{k,i_k} \right| \right\|_{\infty,B_{\gamma_k}} \right) \left\| x \right\|\\
& \lesssim \left\| x \right\|.
\end{align*}
Now taking the average on $(t_1,\ldots,t_d)$, we deduce that
\begin{equation*}
N_d\left([x_{i_1,\ldots,i_d}]\right) \lesssim \left\| x \right\|.
\end{equation*}
The same method yields a similar estimate $N_d([x_{i_1,\ldots,i_d}^*])\lesssim \left\| x^* \right\|$. 
We deduce an estimate
\begin{equation*}
\left| \langle x^*, h_n(T_1,\ldots,T_d)x \rangle \right| \lesssim 
\left\| h \right\|_{\infty,B_{\gamma_1'} \times \cdots \times 
B_{\gamma_d'}} \left\| x \right\| \left\| x^* \right\|.
\end{equation*}
Next the Hahn-Banach Theorem yields the inequality (\ref{majhn}).

The same estimate holds true when $(T_1,\ldots,T_n)$ is replaced
by $(rT_1,\ldots,rT_n)$ for any $r\in(0,1)$.
Further the above argument also shows that $(h_n)_{n\geq 1}$ is a bounded sequence of the space 
$H^\infty_0(B_{\gamma'_1} \times \cdots \times B_{\gamma_d'})$. Moreover, 
the sequence $(h_n)_{n\geq 1}$ 
converges pointwise to $h$. Hence applying Lebesgue's dominated 
convergence Theorem twice we have
$$
\lim_{n\to\infty} h_n(rT_1,\ldots,rT_n) = h(rT_1,\ldots,rT_n)
$$
for any $r\in(0,1)$ and
$$
\lim_{r\to 1^{-}} h(rT_1,\ldots,rT_n) = h(T_1,\ldots,T_n).
$$
We therefore deduce from  (\ref{majhn}) that
$$
\left\| h(T_1,\ldots,T_d) \right\| \leq K \left\| h \right\|_{\infty,B_{\gamma_1'} 
\times \cdots \times B_{\gamma_d'}},
$$
which concludes the proof. 
\end{proof}

\begin{proof}[Proof of Proposition \ref{prop33}] 
Let $X$ be a Banach lattice and let $d \geq 2$ be an integer. 
For any integer $n\geq 1$, 
for any family of complex numbers $(a_{i_1,\ldots,i_d})_{1 \leq i_1,\ldots,i_d \leq n}$ 
and for any families $(x_{i_1,\ldots,i_d})_{1 \leq i_1,\ldots,i_d \leq n}$ in $X$ and 
$(x_{i_1,\ldots,i_d}^*)_{1 \leq i_1,\ldots,i_d \leq n}$ in $X^*$, we have
\begin{align*}
\biggl\vert \sum a_{i_1,\ldots,i_d} & \langle x_{i_1,\ldots,i_d}^*, x_{i_1,\ldots,i_d} 
\rangle \biggr\vert \\
\leq\, &\sup \left\{ \left|a_{i_1,\ldots,i_d} \right| \right\} \left\| 
\left( \sum \left| x_{i_1,\ldots,i_d}\right|^2 \right)^\frac{1}{2} 
\right\|_X  \left\| \left( \sum \left| x_{i_1,\ldots,i_d}^* 
\right|^2 \right)^\frac{1}{2} \right\|_{X^*},
\end{align*}
where $\left( \sum \left| x_{i_1,\ldots,i_d} \right|^2 \right)^\frac{1}{2}$ 
and $\left( \sum \left| x_{i_1,\ldots,i_d}\right|^2 \right)^\frac{1}{2} $ 
are defined in \cite[Section 1.d]{LT2}. 
This follows from basic properties of Krivine's functional calculus on Banach lattices.

By the $d$-variable Khintchine inequality, there exists a constant $C>0$ 
(not depending on the $x_{i_1,\ldots,i_d}$)
such that we have an inequality 
$$
\left( \sum \left| x_{i_1,\ldots,i_d} \right|^2 \right)^\frac{1}{2}
\,\leq\, C\,\int_{\Omega_0^d}
\Bigl\vert\sum  r_{i_1}(t_1)\ldots r_{i_d}(t_d)
\,  x_{i_1,\ldots,i_d}\Bigr\vert\, d\Pdb^d(t_1,\ldots,t_d)
$$
in $X$. By the triangle inequality, this implies that 
$$
\left\| \left( \sum \left| x_{i_1,\ldots,i_d}\right|^2 \right)^\frac{1}{2} 
\right\|_X  \leq C N_d\left([x_{i_1,\ldots,i_d}]\right).
$$
Likewise, we have 
$$
\left\| \left( \sum \left| x_{i_1,\ldots,i_d}^{*}\right|^2 \right)^\frac{1}{2} 
\right\|_{X^*}  \leq C N_d\left([x_{i_1,\ldots,i_d}^{*}]\right).
$$
Combining these three estimates we obtain that $X$ satisfies property $(A_d)$. 
\end{proof}

Before giving the proof of Proposition \ref{prop34}, we show that any Banach space 
with property $(\alpha)$ verifies a
$d$-variable version of (\ref{alpha}).

\begin{lemma}\label{lemalpha}
Let $X$ be a Banach space with property $(\alpha)$. 
For any integer $d \geq 2$, there exists
a constant $C>0$ such that for any integer 
$n\geq 1$, any family $(a_{i_1,\ldots,i_d})_{1 \leq i_1,\ldots,i_d \leq n}$ 
of complex numbers and any family $(x_{i_1,\ldots,i_d})_{1 \leq i_1,\ldots,i_d \leq n}$
in $X$,
\begin{equation}\label{Alpha-d}
N_d \left([a_{i_1,\ldots,i_d} x_{i_1,\ldots,i_d}]\right) \leq  C 
\sup_{1\leq i_1,\ldots,i_d\leq n}
\left\lbrace \left| a_{i_1,\ldots,i_d} \right| \right\rbrace N_d\left([x_{i_1,\ldots,i_d}]\right). 
\end{equation}
\end{lemma}

\begin{proof}
According to \cite[Remark 2.1]{P2}, 
property $(\alpha)$ is equivalent to the fact
that the linear mapping 
$$
\sum_{i,j} r_{i,j} \otimes x_{i,j}  \mapsto  \sum_{i,j} r_i \otimes r_j 
\otimes x_{i,j}
$$
induces an isomorphism from $\text{Rad}({\mathbb{N}^{*2}};X)$ onto 
$\text{Rad}(\text{Rad}(X)) = \text{Rad}^2(X)$. 
This readily implies that for any 
countable sets $I_1,I_2$, we have a natural isomorphism
$$
\text{Rad}(I_1\times I_2;X) \approx \text{Rad}(I_1; \text{Rad}(I_2;X))
$$
when $X$ has property $(\alpha)$.

Under this assumption, we thus have 
\begin{equation*}
\text{Rad}(\text{Rad}(\mathbb{N}^{*2};X)) \approx \text{Rad}(\mathbb{N}^* \times 
\mathbb{N}^{*2};X) = \text{Rad}(\mathbb{N}^{*3};X)
\end{equation*}
and
$$
\text{Rad}(\text{Rad}(\mathbb{N}^{*2};X)) \approx \text{Rad}(\text{Rad}(\text{Rad}(X)))
=\text{Rad}^3(X),
$$
whence a natural isomorphism
\begin{equation*}
\text{Rad}(\mathbb{N}^{*3};X) \approx \text{Rad}^3(X).
\end{equation*}
Proceeding by induction, we obtain that
\begin{equation*}
\text{Rad}(\mathbb{N}^{*d};X) \approx \text{Rad}^d(X).
\end{equation*}
This means that for finite families $(x_{i_1,\ldots,i_d})$ of $X$,
$N_d([x_{i_1,\ldots,i_d}])$ and $\left\| 
\sum r_{i_1,\ldots,i_d} \otimes x_{i_1,\ldots,i_d} \right\|$ are equivalent.
Now recall that by the unconditionality property of Rademacher averages,
\begin{align*}
\Biggnorm{\sum_{1 \leq i_1,\ldots,i_d \leq n} a_{i_1,\ldots,i_d} &
r_{i_1,\ldots,i_d} \otimes   x_{i_1,\ldots,i_d}}_{{\rm Rad}(\Ndb^{*d};X)}\\
\leq & 2 
\,\sup \left\lbrace | a_{i_1,\ldots,i_d} | \right\rbrace 
\Biggnorm{\sum_{1 \leq i_1,\ldots,i_d \leq n} r_{i_1,\ldots,i_d} 
\otimes x_{i_1,\ldots,i_d} }_{{\rm Rad}(\Ndb^{*d};X)}
\end{align*}
for every finite family $(a_{i_1,\ldots,i_d})$ of complex numbers. 
The inequality (\ref{Alpha-d}) follows at once.
\end{proof}

\begin{proof}[Proof of Proposition \ref{prop34}] Assume that $X$ has property 
$(\alpha)$. Let $(x_{i_1,\ldots,i_d})$, $(x_{i_1,\ldots,i_d}^*)$ and $(a_{i_1,\ldots,i_d})$ 
be finite families of $X$, $X^*$ and $\Cdb$, respectively, indexed by 
$(i_1,\ldots,i_d)\in\Ndb^{*d}$.

By the independence of Rademacher variables, we have
\begin{align*} \sum_{i_1,\ldots,i_d} & a_{i_1,\ldots,i_d} \langle x_{i_1,\ldots,i_d}^*, 
x_{i_1,\ldots,i_d} \rangle =\\  & \int_{\Omega_0^d} \left\langle \sum_{i_1,\ldots,i_d}  
r_{i_1}(t_1) \cdots r_{i_d}(t_d) x_{i_1,\ldots,i_d}^*, \sum_{i_1,\ldots,i_d} 
a_{i_1,\ldots,i_d} r_{i_1}(t_1) \cdots r_{i_d}(t_d)x_{i_1,\ldots,i_d} 
\right\rangle d\mathbb{P}^d(t_1,\ldots,t_d).
\end{align*}
By the Cauchy-Schwarz inequality, this implies that
\begin{align*}
 \left| \sum_{i_1,\ldots,i_d} a_{i_1,\ldots,i_d} \langle x_{i_1,\ldots,i_d}^*, 
x_{i_1,\ldots,i_d} \rangle \right|  & \leq \left\| \sum_{i_1,\ldots,i_d} r_{i_1} \otimes 
\cdots \otimes r_{i_d} \otimes x_{i_1,\ldots,i_d}^* \right\|_{{\rm Rad}^d(X^*)} \\ 
 & \times \left\| \sum_{i_1,\ldots,i_d} a_{i_1,\ldots,i_d} r_{i_1} \otimes \cdots \otimes 
r_{i_d} \otimes x_{i_1,\ldots,i_d} \right\|_{{\rm Rad}^d(X)}.
\end{align*} 
By Lemma \ref{lemalpha}, we deduce an estimate
\begin{align*}
\biggl\vert\sum_{i_1,\cdots,i_d} & a_{i_1,\cdots,i_d} \langle x_{i_1,\cdots,i_d}^*, 
x_{i_1,\cdots,i_d} \rangle \biggr\vert  
\leq C \sup \left\{\left| a_{i_1,\cdots,i_d} \right| \right\}\\ 
& \times\left\| \sum_{i_1,\cdots,i_d} r_{i_1} 
\otimes \cdots \otimes r_{i_d} \otimes x_{i_1,\cdots,i_d}^* \right\|_{{\rm Rad}^d(X^*)} 
\left\| \sum_{i_1,\cdots,i_d} 
r_{i_1} \otimes \cdots \otimes r_{i_d}\otimes x_{i_1,\cdots,i_d} \right\|_{{\rm Rad}^d(X)}, 
\end{align*}
which proves $(A_d)$.

The same proof holds true if $X^*$ verifies the property $(\alpha)$. 
\end{proof}

\section{Characterisation by dilation on UMD spaces with property $(\alpha)$}\label{Dilations}
In this section, we give characterisations of $H^\infty$ 
joint functional calculus for
commuting families of either Ritt or sectorial operators
acting on a UMD Banach space $X$ with property $(\alpha)$. 
We pay a special attention to the case when $X$ in an $L^p$-space, 
for $p\in(1,\infty)$.
These characterisations generalise some of 
the main results of \cite{AFLM}.

We refer the reader  to \cite{Bu} and to \cite[Chapter 5]{P3} for 
information on the UMD property.

We first establish a general result about combining dilations of commuting
operators through Bochner spaces. 
Given any $p \in [1,\infty)$, any measure space 
$\Omega$, any Banach space $X$, and any bounded operators
$T\colon L^p(\Omega)\to L^p(\Omega)$ and 
$S\colon X\to X$, consider the operator $T \otimes S$ 
acting on $L^p(\Omega) \otimes X$. If this operator extends to a 
bounded operator on $L^p(\Omega;X)$,
we denote this extension by
\begin{equation*}
T \overline{\otimes} S \colon L^p(\Omega ; X) \longrightarrow L^p(\Omega ; X).
\end{equation*}
By the density of $L^p(\Omega) \otimes X$ in $L^p(\Omega ; X)$, this extension
is necessarily unique.
We recall that if $T$ is a positive operator (meaning that $T(x) \geq 0$ for every $x \geq 0$), then
$T \otimes S$ has a bounded  extension as described above.

\begin{lemma} \label{lem41}
Let $d \geq 2$ be an integer, let $T_1,\ldots,T_d$ be commuting operators on a Banach space $X$ and 
let $p \in [1,\infty)$. Let $1 \leq m \leq d$. Assume that:

\begin{itemize}
\item[(1)] For every $k=1,\ldots,m$, there exist a positive operator $V_k$ on some $L^p(\Omega)$  
and two bounded operators $J_k \colon X \to L^p(\Omega;X)$ and $Q_k \colon L^p(\Omega;X) \to X$ such that
\begin{equation} \label{dil}
T_k^{n_k} = Q_k (V_k \overline{\otimes} I_X)^{n_k} J_k, \qquad n_k \in \mathbb{N}.
\end{equation}

\item[(2)] If $m<d$, there exist a Banach space $Y$, two bounded operators 
$J_{m+1} \colon X \to Y$ and $Q_{m+1} \colon Y \to X $ 
as well as commuting bounded operators $V_{m+1},\ldots,V_d$ on $Y$ such that
\begin{equation} \label{dil2}
T_{m+1}^{n_{m+1}} \cdots T_d^{n_d} = Q_{m+1} V_{m+1}^{n_{m+1}} 
\cdots V_d^{n_d} J_{m+1}, \qquad (n_{m+1},\ldots,n_d) \in \mathbb{N}^{d-m}.
\end{equation} 

\item[(3)] For every $i=1,\ldots,m$ and $j=1,\ldots,d$, we have
\begin{equation} \label{comJT}
J_i T_j = (I_{L^p(\Omega)} \overline{\otimes} T_j) J_i.
\end{equation}
\end{itemize}
Then there exist 
two bounded operators $J \colon X \to L^p(\Omega^m;Y)$ and $Q \colon L^p(\Omega^m;Y) \to X$ such that
\begin{equation} \label{combdil}
T_1^{n_1} \cdots T_d^{n_d} = Q U_1^{n_1}\cdots U_d^{n_d} J, 
\qquad (n_1,\ldots,n_d) \in \mathbb{N}^d,
\end{equation} 
where the operators $U_1,\ldots,U_d\colon L^p(\Omega^m;Y)\to L^p(\Omega^m;Y)$ are given by
\begin{equation} \label{Rk}
U_k = I^{\otimes k-1} ~\overline{\otimes}~ V_k ~\overline{\otimes} 
I^{\otimes m-k} ~\overline{\otimes} ~I_Y, \qquad k=1,\ldots,m;
\end{equation}
\begin{equation} \label{Rk2}
U_k = I^{\otimes m} \overline{\otimes} ~ V_k, \qquad k=m+1,\ldots,d.
\end{equation}
Here $I = I_{L^p(\Omega)}$ and $I^{\otimes l}
=\underbrace{I \overline{\otimes} \cdots \overline{\otimes} I}_{l\ {\rm factors}}$
for every integer $l \geq 1$.
\end{lemma}

\begin{proof}
We define $\widetilde{Q_m} \colon L^p(\Omega^m;X)\to X$ and 
$\widetilde{J_m}\colon X\to L^p(\Omega^m;X)$ by letting
\begin{equation} \label{Q}
\widetilde{Q_m} = Q_1 (I \overline{\otimes} Q_2)(I^{\otimes 2} 
\overline{\otimes} Q_3)\cdots (I^{\otimes m-1} \overline{\otimes} Q_m)
\end{equation} 
and
\begin{equation} \label{J}
\widetilde{J_m} = (I^{\otimes m-1} \overline{\otimes} J_m) 
\cdots (I^{\otimes 2} \overline{\otimes} J_3) (I \overline{\otimes} J_2) J_1.
\end{equation}
Then we define $S_{k,m}\colon L^p(\Omega^m;X)\to L^p(\Omega^m;X)$ by
\begin{equation} \label{Sk}
S_{k,m} = I^{\otimes k-1} ~\overline{\otimes}~ V_k 
~\overline{\otimes} I^{\otimes m-k} ~\overline{\otimes} ~I_X, \qquad 1 \leq k \leq m.
\end{equation}
Our first aim is to prove by induction on $m$ that 
we have the following dilation property,
\begin{equation}\label{premcomb}
T_1^{n_1}\cdots T_m^{n_m} = \widetilde{Q_m} S_{1,m}^{n_1} \cdots S_{m,m}^{n_m} \widetilde{J_m}, 
\qquad (n_1,\ldots,n_m) \in \mathbb{N}^m.
\end{equation}
We will see that this property only depends on the
assumptions (\ref{dil}) and (\ref{comJT}).

The case $m=1$ is trivial. Let $m\geq 2$, suppose  that 
(\ref{Q}), (\ref{J}), (\ref{Sk}) and (\ref{premcomb}) 
hold true for $m - 1$, and let us prove the latter dilation
 property for $m$. 
For every $(n_1,\ldots,n_{m}) \in \mathbb{N}^{m}$, we write 
\begin{equation} \label{Rec}
T_1^{n_1} \cdots T_{m-1}^{n_{m-1}} T_{m}^{n_{m}} 
= \widetilde{Q_{m-1}} S_{1,m-1}^{n_1}\cdots S_{m-1,m-1}^{n_{m-1}} 
\widetilde{J_{m-1}} T_{m}^{n_{m}}.
\end{equation}
We compute the last term $\widetilde{J_{m-1}} T_{m}^{n_{m}}$.
First by (\ref{comJT}), we have
\begin{equation*}
J_1 T_{m}^{n_{m}} = (I \overline{\otimes} T_{m}^{n_{m}})J_1.  
\end{equation*}
Applying (\ref{comJT}) again, we then have
\begin{equation*}
(I \overline{\otimes} J_2) (I \overline{\otimes} T_{m}^{n_{m}})
J_1 = (I^{\otimes 2} \overline{\otimes} T_{m}^{n_{m}}) (I \overline{\otimes} J_2)J_1.
\end{equation*}
Repeating this process with each factor of $\widetilde{J_{m-1}}$, we obtain
\begin{equation}\label{mm}
\widetilde{J_{m-1}} T_{m}^{n_{m}} = (I^{\otimes m-1} \overline{\otimes} T_{m}^{n_{m}}) 
\widetilde{J_{m-1}}.
\end{equation}
Using (\ref{dil}) for $T_{m}$, we see that
\begin{equation*}
I^{\otimes m-1} \overline{\otimes} T_{m}^{n_{m}} = 
(I^{\otimes m-1} \overline{\otimes} Q_{m})(I^{\otimes m-1} \overline{\otimes} V_{m} 
\overline{\otimes} I_X)^{n_{m}} (I^{\otimes m-1} \overline{\otimes} J_{m}).
\end{equation*}
Combining with (\ref{Rec}) and (\ref{mm}), and using the fact that 
$I^{\otimes m-1} \overline{\otimes} V_{m} 
\overline{\otimes} I_X = S_{m,m}$, we deduce that
$$
T_1^{n_1}\cdots T_m^{n_m} =
\widetilde{Q_{m-1}} S_{1,m-1}^{n_1}\cdots S_{m-1,m-1}^{n_{m-1}}
(I^{\otimes m-1}\overline{\otimes} Q_m) S_{m,m}^{n_m} (I^{\otimes m-1}\overline{\otimes}
J_m)\widetilde{J_{m-1}}.
$$
A thorough look at (\ref{Sk}) reveals that for any $k=1,\ldots,m-1$,
\begin{equation*}
S_{k,m-1} (I^{\otimes m-1} \overline{\otimes} Q_{m}) = 
(I^{\otimes m-1} \overline{\otimes} Q_{m}) S_{k,m}.
\end{equation*}
Consequently
\begin{equation*}
T_1^{n_1} \cdots T_{m}^{n_{m}} = 
\widetilde{Q_{m-1}} (I^{\otimes m-1} \overline{\otimes} Q_{m})S_{1,m}^{n_1}
\cdots S_{m-1,m}^{n_{m-1}} S_{m,m}^{n_m} 
(I^{\otimes m-1}\overline{\otimes} J_m)\widetilde{J_{m-1}}.
\end{equation*}
Since 
$$
\widetilde{Q_{m}} =  \widetilde{Q_{m-1}} (I^{\otimes m-1} \overline{\otimes} Q_{m})
\qquad\hbox{and}\qquad
\widetilde{J_{m}}= (I^{m-1}\overline{\otimes} J_m)\widetilde{J_{m-1}}, 
$$
this yields property (\ref{premcomb}).

If $m=d$, the preceding computation proves the lemma. 
Assume now that $m \leq d-1$. It follows from 
(\ref{premcomb}) that
for any $(n_1,\ldots,n_d) \in \mathbb{N}^d$, we have
$$
T_1^{n_1}\cdots T_{d}^{n_d} = \widetilde{Q_m} S_{1,m}^{n_1} 
\cdots S_{m,m}^{n_m} \widetilde{J_m} T_{m+1}^{n_{m+1}} \cdots T_d^{n_d}.
$$
Using (\ref{comJT}) we obtain that for any $k=m+1,\ldots,d$,
\begin{equation} \label{comJT2}
\widetilde{J_m} T_{k}^{n_k} = (I^{\otimes m} \overline{\otimes} T_k)^{n_k} \widetilde{J_m}.
\end{equation}
Applying (\ref{dil2}), we therefore obtain that

\begin{align*}
T_1^{n_1}\cdots T_{d}^{n_d} & = \widetilde{Q_m} S_{1,m}^{n_1} \cdots S_{m,m}^{n_m}
\label{decomp} \\
& \times (I^{\otimes m} \overline{\otimes} Q_{m+1}) (I^{\otimes m} \overline{\otimes} 
V_{m+1})^{n_{m+1}} \cdots (I^{\otimes m} \overline{\otimes} V_{d})^{n_d} 
(I^{\otimes m} \overline{\otimes} J_{m+1}) \widetilde{J_m}.
\end{align*}
Using (\ref{Rk2}), this yields
\begin{equation}\label{factor}
T_1^{n_1}\cdots T_{d}^{n_d}  = \widetilde{Q_m} S_{1,m}^{n_1} \cdots S_{m,m}^{n_m}
(I^{\otimes m} \overline{\otimes} Q_{m+1}) U_{m+1}^{n_{m+1}} \cdots U_{d}^{n_d} 
(I^{\otimes m} \overline{\otimes} J_{m+1}) \widetilde{J_m}.
\end{equation}
Now it follows from 
(\ref{Sk}) that for any $k=1,\ldots,m$,
\begin{equation}\label{Uk}
S_{k,m}^{n_k} (I^{\otimes m} \overline{\otimes} Q_{m+1}) = 
(I^{\otimes m} \overline{\otimes} Q_{m+1}) U_k,
\end{equation}
where the $U_k$ are given by (\ref{Rk}). 
Set 
$$
Q = \widetilde{Q_m} (I^{\otimes m} \overline{\otimes} Q_{m+1})
\qquad\hbox{and}\quad
J = (I^{\otimes m} \overline{\otimes} J_{m+1}) \widetilde{J_m}. 
$$
Then (\ref{combdil}) follows from the factorisation (\ref{factor})
and the relation (\ref{Uk}).
\end{proof}

The following result is a $d$-variable version of \cite[Theorem 4.1]{AFLM}.
We refer the reader to \cite[Chapter 11]{DJT} for the definitions 
and basic properties of spaces with finite cotype.

\begin{theorem} \label{Th42}
Let $X$ be a reflexive Banach space such that $X$ and $X^*$ 
have finite cotype. Let $T_1,\ldots,T_d$ be commuting Ritt operators 
on $X$ such that every $T_k$ has an $H^\infty(B_{\gamma_k})$ functional 
calculus for some $\gamma_k \in (0,\frac{\pi}{2})$. Let $p\in(1,\infty)$. 
Then there exist a measure space $\Omega$,
commuting isometric isomorphisms $U_1,\ldots,U_d$ on $L^p(\Omega;X)$,
and two bounded operators 
$J \colon X \to L^p(\Omega;X)$ and $Q \colon L^p(\Omega;X) \to X$ 
such that
\begin{equation}
T_1^{n_1}\cdots T_d^{n_d} = Q U_1^{n_1} \cdots U_d^{n_d} J, 
\qquad (n_1,\ldots,n_d) \in \mathbb{N}^d.
\end{equation}
\end{theorem}

\begin{proof} 
We shall apply Lemma \ref{lem41} in the case $m=d$,
using the construction devised
in the proof of \cite[Theorem 4.1]{AFLM}. 

We recall this construction. Following Section \ref{Automatic}, we let
$(r_n)_{n\in\footnotesize{\Zdb}}$ be an independent sequence of Rademacher variables 
on some probability space $\Omega_0$.

For any $k=1,\ldots,d$, recall the ergodic
decomposition
$X = \text{Ker}(I-T_k) \oplus \overline{\text{Ran}(I-T_k)}.$
It is shown in \cite{AFLM} that the operator
\begin{equation}
J_k \colon \begin{array}{ccl} X = \text{Ker}(I-T_k) \oplus 
\overline{\text{Ran}(I-T_k)} & \to & X \oplus_p L^p(\Omega_0;X)
\\ x_0+x_1 & \mapsto & 
\bigl(x_0, \sum_{n=1}^{\infty} r_n \otimes T_k^n(I-T_k)^\frac{1}{2}(I+T_k) (x_1)\bigr) 
\end{array}
\end{equation}
is well-defined and bounded, under the assumption that 
$T_k$ has an $H^\infty(B_{\gamma_k})$ functional calculus 
for some $\gamma_k \in (0,\frac{\pi}{2})$.
More precisely, the series 
$$
\sum_{n=1}^{\infty} r_n \otimes T_k^n(I-T_k)^\frac{1}{2}(I+T_k) (x_1)
$$
converges in $L^p(\Omega_0;X)$ for any $x_1\in X$ and the norm of the resulting
sum is $\lesssim\norm{x_1}$. 

Define $\Omega$ as the disjoint union 
of $\Omega_0$ and a singleton, so that  
$$
X \oplus_p L^p(\Omega_0;X) \simeq L^p(\Omega;X).
$$ 
It also follows from the proof 
of \cite[Theorem 4.1]{AFLM} that there exist an isometric isomorphism
$U\colon L^p(\Omega)\to L^p(\Omega)$ (which does not depend on $k$)
and operators $Q_k\colon L^p(\Omega;X)\to X$ such that
$$
T_k^{n_k} = Q_k(U\overline{\otimes} I_X)^{n_k}J_k,\qquad n_k\in\Ndb.
$$
We set $V_k=U$ for any $k=1,\ldots,d$, so that $T_1,\ldots, T_d$ satisfy 
(\ref{dil}).

Let us show that $T_1,\ldots, T_d$ also satisfy 
(\ref{comJT}). Consider arbitrary $i,j$ in $\{1,\ldots,d\}$, and
an element $x_0+x_1 \in X = 
\text{Ker}(I-T_i) \oplus \overline{\text{Ran}(I-T_i)}$. 
Since $T_i$ and $T_j$ commute, $T_j(x_0)$ belongs to $\text{Ker}(T_i)$.
Consequently, 
\begin{align*}
J_i(T_j(x_0+x_1)) 
& = \bigl(T_j(x_0), \sum_{n=1}^{\infty} r_n 
\otimes T_i^n(I-T_i)^\frac{1}{2}(I+T_i) T_j (x_1)\bigr)\\
& = \bigl(T_j(x_0), \sum_{n=1}^{\infty} r_n \otimes T_j 
T_i^n(I-T_i)^\frac{1}{2}(I+T_i)(x_1)\bigr)\\
& = \Bigl(T_j(x_0), \left(I_{L^p(\Omega_0)} 
\overline{\otimes} T_j \right) 
\Bigl(\sum_{n=1}^{\infty} r_n 
\otimes (T_i^n(I-T_i)^\frac{1}{2}(I+T_i)(x_1)\Bigr)\Bigr)\\
& = \left(I_{L^p(\Omega)} \overline{\otimes} T_j\right) J_i (x_0+x_1).
\end{align*}
This proves (\ref{comJT}).

Applying Lemma \ref{lem41}, we deduce the existence of  
two bounded operators $Q \colon L^p(\Omega^d;X) \to X$ and 
$J \colon X \to L^p(\Omega^d;X)$ such that
\begin{equation*}
T_1^{n_1} \cdots T_d^{n_d} = Q U_1^{n_1} \cdots U_d^{n_d} J, 
\qquad (n_1,\ldots,n_d) \in \mathbb{N}^d,
\end{equation*}
where $U_1,\ldots,U_d$ are given by 
\begin{equation*}
U_k = I^{\otimes k-1} ~\overline{\otimes}~ U 
~\overline{\otimes} I^{\otimes d-k-1} ~\overline{\otimes} ~I_X.
\end{equation*}
Since $U$ is an isometric isomorphism of $L^p(\Omega)$, it is clear that each
$U_k$ is an isometric isomorphism as well. 
\end{proof}

We are now in position
to extend \cite[Theorem 5.1]{AFLM} to $d$-tuples of Ritt operators.

\begin{theorem} \label{Th43}
Let $X$ be a UMD Banach space with property 
$(\alpha)$  
and let $d\geq 1$ be an integer. 
Let $T_1,\ldots,T_d$ be commuting Ritt operators on $X$ and 
let $p\in(1,\infty)$. The following two conditions are equivalent.
\begin{itemize}
\item[(1)] $(T_1,\ldots,T_d)$ admits 
an $H^\infty(B_{\gamma_1} \times \cdots 
\times B_{\gamma_d})$ joint functional calculus 
for some $\gamma_k \in (0,\frac{\pi}{2})$, $k=1,\ldots,d$.
\item[(2)] There exist a measure space $\Omega$, commuting contractive 
Ritt operators $R_1,\ldots,R_d$ on $L^p(\Omega;X)$ such that every $R_k$ 
admits an $H^\infty(B_{\gamma_k'})$ functional calculus for some 
$\gamma_k' \in (0,\frac{\pi}{2})$, $k=1,\ldots,d$,
as well as two bounded operators $J \colon X \to L^p(\Omega;X)$
and $Q \colon L^p(\Omega;X) \to X$ such that   
\begin{equation}\label{Th43-Dil}
T_1^{n_1} \cdots T_d^{n_d} = Q R_1^{n_1} \cdots R_d^{n_d} J, 
\qquad (n_1,\ldots,n_d) \in \mathbb{N}^d.
\end{equation}
\end{itemize}
\end{theorem}

\begin{proof} The implication 
``$(2) \Rightarrow (1)$" is easy. Indeed  (\ref{Th43-Dil}) implies
that for any $\phi\in\P_d$ (the algebra of complex 
polynomials in $d$ variables), we have
\begin{equation*}
\phi(T_1,\ldots,T_d) = Q \phi(R_1,\ldots,R_d) J,
\end{equation*}
and hence
\begin{equation*}
\left\|\phi(T_1,\ldots,T_d)  \right\| \leq \left\| 
Q \right\| \left\| J \right\| \left\| \phi(R_1,\ldots,R_d) 
\right\|.
\end{equation*}
By assumption each $R_k$ has an $H^\infty(B_{\gamma_k'})$ functional calculus, with
$\gamma_k'\in(0,\frac{\pi}{2})$. Since $X$ 
has property $(\alpha)$, the Bochner space $L^p(\Omega;X)$ has property
$(\alpha)$ as well. It therefore follows from Theorem \ref{Th32} that the $d$-tuple 
$(R_1,\ldots,R_d)$ has an
$H^\infty(B_{\gamma_1} \times \cdots 
\times B_{\gamma_d})$ joint functional calculus 
for some $\gamma_k \in (0,\frac{\pi}{2})$.
Applying Proposition \ref{Rittpolynome}, we deduce that $(T_1,\ldots,T_d)$ 
also has an $H^\infty(B_{\gamma_1} \times \cdots 
\times B_{\gamma_d})$ joint functional calculus.

To prove the converse (and main) implication ``$(1) \Rightarrow (2)$", we
assume (1). Every UMD Banach space 
is reflexive and has finite cotype, so we can apply Theorem \ref{Th42} on $X$.

As in \cite[Section 3]{AFLM}, set 
$$
(T_k)_a =I_X - (I_X- T_k)^{a},\qquad a>0.
$$
Since $(T_1,\ldots,T_d)$ has an $H^\infty$ joint functional calculus, 
every $T_k$ has an $H^\infty$ functional calculus.
Hence according to \cite[Proposition 3.2]{AFLM}, there exists $a>1$  
such that every $(T_k)_a$ has an $H^\infty$ functional calculus. 
Applying Theorem \ref{Th42}, we deduce a dilation property
\begin{equation*}
((T_1)_a)^{n_1} \cdots ((T_d)_a)^{n_d}= Q U_1^{n_1} 
\cdots U_d^{n_d} J, \qquad (n_1,\ldots,n_d) \in \mathbb{N}^d,
\end{equation*}
where $J \colon X \to L^p(\Omega;X)$ and $Q\colon L^p(\Omega;X) \to X$ 
are bounded operators and $U_1,\ldots,U_d$ are isometric isomorphisms
on $L^p(\Omega;X)$ .

Let $b = \frac{1}{a}$, so that $0<b<1$. 
Arguing as in the proof of \cite[Theorem 5.1]{AFLM} (see also 
\cite{Fa}, where this argument appeared for the first time), 
we derive that
\begin{equation*}
T_1^{n_1} \cdots T_d^{n_d}= Q ((U_1)_b)^{n_1} \cdots ((U_d)_b)^{n_d} J, 
\qquad (n_1,\ldots,n_d) \in \mathbb{N}^d.
\end{equation*}
We let $R_k = (U_k)_b$ for every $k=1,\ldots,d$. 
By \cite[Theorem 3.1 and 3.3]{AFLM}, and the assumption that $X$
is a UMD Banach space, every $R_k$ is a contractive 
Ritt operator having an $H^\infty(B_{\gamma_k'})$ functional 
calculus for some $\gamma_k' \in (0,\frac{\pi}{2})$, which proves (2).
\end{proof}

\begin{remark}\label{Positive}
It follows from the proof of \cite[Theorem 4.1]{AFLM} that the
isometric isomorphism $U\colon L^p(\Omega)\to L^p(\Omega)$ appearing
in the proof of Theorem \ref{Th42} is positive. This implies that 
if $X$ is an ordered Banach space, then the isometric isomorphisms 
$U_1,\ldots,U_d\colon L^p(\Omega;X)\to L^p(\Omega;X)$ in the latter theorem
are positive operators.
It therefore follows from \cite[Theorem 3.1 (c)]{AFLM} that 
if $X$ is an ordered Banach space in Theorem \ref{Th43}, then the 
contractive Ritt operators $R_1,\ldots,R_d\colon L^p(\Omega;X)\to L^p(\Omega;X)$ 
in this theorem
are positive operators.
\end{remark}

We note that any UMD Banach lattice has property $(\alpha)$.
Hence any UMD Banach lattice satisfies Theorem \ref{Th43}.

We also observe 
that thanks to Theorem \ref{Th32}, assumption (1) of  Theorem \ref{Th43}
is equivalent to the property that each $T_k$ admits an $H^\infty(B_{\gamma_k})$ functional calculus
for some $\gamma_k\in(0,\frac{\pi}{2})$.

We now give a specific result on $L^p$-spaces. This is a $d$-variable version 
of \cite[Theorem 5.2]{AFLM}.

\begin{theorem} \label{Th44}
Let $\Sigma$ be a measure space and 
let $p \in (1,\infty)$. Let $T_1,\ldots,T_d$ be commuting Ritt operators 
on $L^p(\Sigma)$. The following two conditions are equivalent.
\begin{itemize} 
\item[(1)] $(T_1,\ldots,T_d)$ admits an 
$H^\infty(B_{\gamma_1} \times \cdots \times B_{\gamma_d})$ 
joint functional calculus for some $\gamma_k \in (0,\frac{\pi}{2})$, $k=1,\ldots,d$.
\item[(2)] There exist a measure space $\Omega$, commuting positive 
contractive Ritt operators $R_1,\ldots,R_d$ on $L^p(\Omega)$, 
and two bounded operators $J \colon 
L^p(\Sigma) \to L^p(\Omega)$ and $Q \colon L^p(\Omega) \to L^p(\Sigma)$
such that   
\begin{equation*}
T_1^{n_1} \cdots T_d^{n_d} = Q R_1^{n_1} \cdots R_d^{n_d} J, 
\qquad (n_1,\ldots,n_d) \in \mathbb{N}^d.
\end{equation*}
\end{itemize}
\end{theorem}

\begin{proof} 
We apply Theorem \ref{Th43} above with $X = L^p(\Sigma)$, which is a UMD Banach space
with property $(\alpha)$. We note that for any measure space $\Omega$,
$L^p(\Omega;L^p(\Sigma))$ is an $L^p$-space. 
Further conditions (1) in Theorem \ref{Th43} and Theorem \ref{Th44} 
are identical. 

Assuming (1) and applying Theorem \ref{Th43} together
with Remark \ref{Positive}, we obtain condition (2) in
Theorem \ref{Th44}.

The converse implication follows from 
Theorem \ref{Th43} and the fact that 
any positive contractive Ritt operator on an
$L^p$-space  has an $H^\infty(B_{\gamma})$ 
functional calculus for some $\gamma \in (0,\frac{\pi}{2})$.
This result is proved in \cite[Theorem 3.3]{LMX}.
\end{proof}

A celebrated theorem of Akcoglu and Sucheston (see \cite{Akc})
asserts that if $T\colon L^p(\Sigma)\to L^p(\Sigma)$ is 
a positive contraction, with $p\in(1,\infty)$, then there 
exist a measure space $\Sigma'$, an isometric isomorphism
$V\colon L^p(\Sigma')\to L^p(\Sigma')$
and two contractions $J\colon L^p(\Sigma)\to L^p(\Sigma')$
and $Q\colon L^p(\Sigma')\to
L^p(\Sigma)$ such that $T^n=QV^nJ$ for any $n\in\Ndb$.
It is an open problem whether the Akcoglu-Sucheston Theorem
extends to pairs. The question reads as follows.

Consider a commuting
pair $(T_1,T_2)$ of positive contractions on $L^p(\Sigma)$. Does there
exist a commuting pair $(V_1,V_2)$ of 
isometric isomorphisms acting on some $L^p(\Sigma')$, as well
as bounded (or even contractive) operators $J\colon L^p(\Sigma)\to L^p(\Sigma')$
and $Q\colon L^p(\Sigma')\to L^p(\Sigma)$ such that 
$T_1^{n_1}T_2^{n_2}=QV_1^{n_1}V_2^{n_2}J$ for any $(n_1,n_2)
\in\Ndb^2$?

The next result shows 
that the answer is positive if either $T_1$ or
$T_2$ is a Ritt operator. More generally we have the following.

\begin{theorem} \label{Th45}
Let $\Sigma$ be a measure space and let $p\in(1,\infty)$. 
Let $T_1,\ldots,T_d$ be commuting positive contractions on $L^p(\Sigma)$. 
Assume further that $T_1,\ldots,T_{d-1}$ are Ritt operators.

Then there exist a measure space $\Omega$, two bounded operators 
$J \colon L^p(\Sigma)\to L^p(\Omega)$ and $Q \colon L^p(\Omega)\to L^p(\Sigma)$,
as well as commuting isometric isomorphisms 
$U_1,\ldots,U_d\colon L^p(\Omega)\to L^p(\Omega)$ such that
\begin{equation*}
T_{1}^{n_1} \cdots T_{d}^{n_d} = Q U_{1}^{n_1} \cdots U_{d}^{n_d} J, 
\qquad (n_1,\ldots,n_d) \in \mathbb{N}^d.
\end{equation*}
\end{theorem}

\begin{proof}
We aim at applying Lemma \ref{lem41} with $m=d-1$ and $X=L^p(\Sigma)$.
For any $k=1,\ldots,d-1$, $T_k$ is a positive Ritt contraction on $L^p(\Sigma)$. 
According to \cite[Theorem 3.3]{LMX},
this implies that it has an $H^\infty(B_{\gamma_k})$ functional calculus 
for some $\gamma_k \in (0,\frac{\pi}{2})$.
By \cite[Theorem 4.1]{AFLM} and its proof, this implies that 
$T_1,\ldots,T_{d-1}$ satisfy the assumption (1) of Lemma \ref{lem41}.

According to the Ackoglu-Sucheston Theorem quoted above, 
$T_d$ satisfies the assumption (2) of Lemma \ref{lem41}, with $Y=L^p(\Sigma')$.

Moreover the argument in the proof of Theorem \ref{Th42} shows that  
$(T_1,\ldots,T_d)$ verifies the assumption (3) of Lemma \ref{lem41}.

The result now follows from this lemma and the fact that
$L^p(\Omega^m;Y)=L^p(\Omega^{d-1};L^p(\Sigma'))$ is an $L^p$-space. Details are left to the reader.
\end{proof}

In the last part of this section, we give analogues of our previous results for sectorial
operators and semigroups. Since the proofs are similar to the ones in the
discrete case, we will be deliberately brief.

We refer the reader to e.g. \cite{Paz} for definitions and basic properties of 
$C_0$-semigroups and bounded analytic semigroups. We recall that if $(T_t)_{t\geq 0}$
is a $C_0$-semigroup on $X$, with generator $-A$, then $A$ is sectorial
of type $<\frac{\pi}{2}$ if and only if  $(T_t)_{t\geq 0}$ is a bounded analytic semigroup.

We say that two $C_0$-semigroups
$(T_{1,t})_{t\geq 0}$ and $(T_{2,t})_{t\geq 0}$ on $X$ 
commute provided that 
\begin{equation}\label{comm}
T_{1,t_1} T_{2,t_2} = T_{2,t_2} T_{1,t_1}, \qquad t_1\geq 0,t_2\geq 0.
\end{equation}
Assume that $(T_{1,t})_{t\geq 0}$ and $(T_{2,t})_{t\geq 0}$ are bounded analytic semigroups
with respective generators $-A_1$ and $-A_2$. Then (\ref{comm}) holds true if and
only if the sectorial operators $A_1,A_2$ commute (in the resolvent sense, see Section \ref{FC}).

It is easy to adapt the proof of Lemma \ref{lem41} 
to semigroups to obtain the following result. We skip the proof.

\begin{lemma}\label{lem46}
Let $d \geq 2$ be an integer, 
let $(T_{1,t})_{t \geq 0},\ldots,(T_{d,t})_{t \geq 0}$ be commuting 
$C_0$-semigroups on a Banach space $X$ and let $p \in [1,\infty)$.
Let $1\leq m\leq d$. Assume that:
\begin{itemize}
\item[(1)] For every $k=1,\ldots,m$, there exist a $C_0$-semigroup $(V_{k,t})_{t \geq 0}$ of 
positive operators on some $L^p(\Omega)$ and two bounded operators 
$J_k \colon X \to L^p(\Omega;X)$ and $Q_k \colon  L^p(\Omega;X) \to X$ such that
\begin{equation*}
T_{k,t} = Q_k (V_{k,t} \overline{\otimes} I_X) J_k, \qquad t \geq 0.
\end{equation*}
\item[(2)] If $m<d$, there exist a Banach space $Y$, two bounded operators 
$J_{m+1}\colon X\to Y$ and $Q_{m+1}\colon Y\to X$ as well as commuting 
$C_0$-semigroups $(V_{m+1,t})_{t\geq 0},\ldots, (V_{d,t})_{t\geq 0}$ on $Y$ such that
$$
T_{m+1,t_{m+1}}\cdots T_{d,t_{d}}=Q_{m+1}V_{m+1,t_{m+1}}\cdots V_{d,t_{d}}
J_{m+1},\qquad t_{m+1}\geq 0,\ldots, t_d\geq 0.
$$
\item[(3)] For every $i=1,\ldots,m$ and $j=1,\ldots,d$, and for any 
$t \geq 0$, we have
\begin{equation*}
J_i T_{j,t} = (I_{L^p(\Omega)} \overline{\otimes} T_{j,t})J_i.
\end{equation*}
\end{itemize}
Then there exist two bounded operators $J \colon X \to L^p(\Omega^m;Y)$ and 
$Q \colon L^p(\Omega^m;Y)\to X$ such that
\begin{equation*}
T_{1,t_1} \cdots T_{d,t_d} = Q U_{1,t_1} \cdots U_{d,t_d} J, 
\qquad t_1 \geq 0,\ldots,t_d \geq 0,
\end{equation*}
where $(U_{1,t})_{t\geq 0},\ldots, (U_{d,t})_{t\geq 0}$ 
are $C_0$-semigroups on $L^p(\Omega^m;Y)$ given by
$$
U_{k,t} = I^{\otimes k-1} \overline{\otimes} V_{k,t} \
\overline{\otimes} I^{\otimes m-k} \overline{\otimes} I_Y,\qquad
k=1,\ldots,m;
$$
$$
U_{k,t}= I^{\otimes m} \overline{\otimes} V_{k,t},\qquad k=m+1,\ldots, d.
$$
\end{lemma}

The construction in the proof of \cite[Theorem 4.5]{AFLM} is an analogue 
of the construction in the proof of \cite[Theorem 4.1]{AFLM} where
discrete square functions based on Rademacher averages are replaced
by continuous square functions provided by Brownian motion. 
Using this construction and using 
Lemma \ref{lem46} instead of Lemma \ref{lem41}, we obtain the following 
sectorial version of Theorem \ref{Th42}.

\begin{theorem} \label{Th46}
Let $X$ be a reflexive Banach space such that $X$ and $X^*$ have finite cotype.
Let $A_1,\ldots,A_d$ be commuting sectorial operators on $X$ such that every 
$A_k$ has an $H^\infty(\Sigma_{\theta_k})$ functional calculus for 
some $\theta_k$ in $(0,\frac{\pi}{2})$. Let $p\in(1,\infty)$.
Then there exist a measure space $\Omega$, commuting $C_0$-groups of isometries 
$(U_{1,t})_{t \in \mathbb{R}},\ldots,(U_{d,t})_{t \in \mathbb{R}}$ on 
$L^p(\Omega;X)$, and two bounded operators $J \colon X \to L^p(\Omega;X)$ and $Q \colon
L^p(\Omega;X) \to X$ such that
\begin{equation*}
e^{-t_1A_1} \cdots e^{-t_dA_d} = Q U_{1,t_1} \cdots U_{d,t_d} J, \qquad t_1 \geq 0,\ldots,t_d \geq 0.
\end{equation*}
\end{theorem}

Using the previous result and adapting 
the proof of \cite[Theorem 5.6]{AFLM} to the $d$-variable case, we obtain
the following 
sectorial version of Theorem \ref{Th43}.

\begin{theorem} \label{Th47}
Let $X$ be a UMD Banach space with property $(\alpha)$ and let $d\geq 1$ be an integer.
Let $A_1,\ldots,A_d$ be commuting sectorial operators 
and let $p \in (1,\infty)$. The following two conditions are equivalent.

\begin{itemize}
\item[(1)] $(A_1,\ldots,A_d)$ admits an 
$H^\infty(\Sigma_{\theta_1} \times \cdots \times \Sigma_{\theta_d})$ 
joint functional calculus for some $\theta_k \in (0,\frac{\pi}{2})$, $k=1,\ldots,d$.
\item[(2)] There exist a measure space $\Omega$, commuting sectorial operators 
$B_1,\ldots,B_d$ on $L^p(\Omega;X)$ such that every $B_k$ 
admits an $H^\infty(\Sigma_{\theta_k'})$ functional calculus for some 
$\theta_k' \in (0,\frac{\pi}{2})$, $k=1,\ldots,d$, as well as 
two bounded operators $J \colon X \to L^p(\Omega;X)$ and $Q 
\colon L^p(\Omega;X) \to X$ such that
\begin{equation*}
e^{-t_1 A_1} \cdots e^{-t_d A_d} = Q e^{-t_1 B_1} 
\cdots e^{-t_d B_d} J, \qquad t_1 \geq 0, \ldots, t_d \geq 0,
\end{equation*}
and all the $(e^{-t B_k})_{t \geq 0}$ are semigroups of contractions.
\end{itemize}
\end{theorem}

We now give the sectorial version 
of Theorem \ref{Th44}.

\begin{theorem} \label{Th48}
Let $\Sigma$ be a measure space and let $p \in (1,\infty)$. 
Let $A_1,\ldots,A_d$ be commuting sectorial operators on $L^p(\Sigma)$. 
The following conditions are equivalent.

\begin{itemize}
\item[(1)] $(A_1,\ldots,A_d)$ admits an $H^\infty(\Sigma_{\theta_1} 
\times \cdots \times \Sigma_{\theta_d})$ joint functional calculus for 
some $\theta_k \in (0,\frac{\pi}{2})$, $k=1,\ldots,d$.
\item[(2)] There exist a measure space $\Omega$, commuting sectorial 
operators $B_1,\ldots,B_d$ on $L^p(\Omega)$ of type $<\frac{\pi}{2}$,
and two bounded operators $J\colon L^p(\Sigma) \to L^p(\Omega)$ and $Q 
\colon L^p(\Omega) \to L^p(\Sigma)$ such that
\begin{equation*}
e^{-t_1 A_1} \cdots e^{-t_d A_d} = Q e^{-t_1 B_1} \cdots e^{-t_d B_d} J, 
\qquad t_1 \geq 0, \ldots, t_d \geq 0,
\end{equation*}
and all the $(e^{-t B_k})_{t \geq 0}$ are semigroups of positive contractions.
\end{itemize}
\end{theorem}

\begin{proof} 
If $B$ is a sectorial operator of type $<\frac{\pi}{2}$ on $L^p(\Omega)$
such that $e^{-t B}$ is a positive contraction for any $t\geq 0$, then 
$B$ has an $H^\infty(\Sigma_\theta)$ functional calculus for some
$\theta<\frac{\pi}{2}$. This result is due to Weis, see \cite{We,KW1}.
Using this and arguing as in the proof of Theorem \ref{Th44}, the result follows at once.
\end{proof}

We conclude with a semigroup version of 
Theorem \ref{Th45}. We first recall that
Fendler \cite{Fe} proved the following semigroup version of the Akcoglu-Sucheston Theorem:
Let $(T_t)_{t\geq 0}$ be a $C_0$-semigroups of positive contractions 
on $L^p(\Sigma)$, with $p\in(1,\infty)$. Then there 
exist a measure space $\Sigma'$, a $C_0$-group $(V_t)_{t\geq 0}$ 
of isometric isomorphisms on $L^p(\Sigma')$
and two contractions $J\colon L^p(\Sigma)\to L^p(\Sigma')$
and $Q\colon L^p(\Sigma')\to
L^p(\Sigma)$ such that $T_t=QV_tJ$ for any $t\geq 0$.

Using this result and Lemma \ref{lem46}, and arguing as in the proof
of Theorem \ref{Th45}, we obtain the following.

\begin{theorem} \label{Th49} 
Let $\Sigma$ be a measure space and let $p\in(1,\infty)$. 
Let $(T_{1,t})_{t\geq 0},\ldots, (T_{d,t})_{t\geq 0}$ be 
$C_0$-semigroups of positive contractions
on $L^p(\Sigma)$. Assume further that $(T_{1,t})_{t\geq 0},\ldots, (T_{d-1,t})_{t\geq 0}$
are bounded analytic semigroups.

Then there exist a measure space $\Omega$, two bounded operators 
$J \colon L^p(\Sigma)\to L^p(\Omega)$ and $Q \colon L^p(\Omega)\to L^p(\Sigma)$,
as well as commuting $C_0$-groups $(U_{1,t})_{t\geq 0},\ldots, (U_{d,t})_{t\geq 0}$
of isometric isomorphisms on $L^p(\Omega)$ such that
\begin{equation*}
T_{1,t_1}  \cdots T_{d,t_d}  = Q U_{1,t_1}  \cdots U_{d,t_d}  J, 
\qquad t_1\geq 0,\ldots, t_d\geq 0.
\end{equation*}
\end{theorem}

\section{The Hilbert space case}\label{Hilbert}
This section is devoted to commuting operators 
on Hilbert space $H$. We will be interested in the following two issues.

First recall that if $T\colon H\to H$ is a Ritt operator, then $T$ has an $H^\infty(B_\gamma)$
functional calculus for some $\gamma<\frac{\pi}{2}$ if and only if
$T$ is similar to a contraction, that is, there exists a bounded invertible 
operator $S\colon H\to H$ such that $S^{-1}TS$ is a contraction on $H$.
This is proved in \cite[Theorem 8.1]{LM}.
We will extend this characterisation to $d$-tuples of Ritt operators,
see Corollary \ref{JFC-Hilbert} below.

Second let $(T_1,\ldots,T_d)$ be a $d$-tuple of commuting contractions
on $H$. If $d=2$, Ando's Theorem \cite{An} (see also \cite[Theorem 1.2]{P}) 
asserts that $\norm{\phi(T_1,T_2)}\leq \norm{\phi}_{\infty,\Ddb^2}$
for any polynomial $\phi\in\P_2$. This result does not extend to $d\geq 3$
and it is unknown whether there exists a universal constant $C\geq 1$ such that 
\begin{equation}\label{vN}
\norm{\phi(T_1,\ldots, T_d)}\leq C\norm{\phi}_{\infty,\Ddb^d}
\end{equation}
for any $\phi\in\P_d$ (see \cite[Chapter 1]{P} for more on this problem). Theorem \ref{Th51} below 
shows that an estimate (\ref{vN}) holds true when at least $d-2$ of these
contractions are Ritt operators.

\begin{theorem} \label{Th51}
Let $d \geq 3$ be an integer and let $H$ be a Hilbert space. 
Let $T_1,\ldots,T_d$ be commuting operators on $H$ such that:
\begin{itemize}
\item[(i)] For every $j$ in $\left\lbrace 1,\ldots,d-2\right\rbrace$, 
$T_j$ is a Ritt operator which is similar to a contraction.
\item[(ii)] There exists a bounded invertible operator 
$S\colon H\to H$ such that $S^{-1} T_{d-1} S$ and $S^{-1} T_{d} S$ 
are both contractions.
\end{itemize}
Then we have the following three properties:
\begin{itemize}
\item[(1)] There exist a Hilbert space $K$, two bounded operators 
$J \colon H \to K$ and $Q \colon K \to H$ and 
commuting unitary operators $U_1,\ldots,U_d$ on $K$ such that
\begin{equation} \label{combcontr}
T_1^{n_1} \cdots T_d^{n_d} = Q U_1^{n_1} \cdots U_d^{n_d} J, 
\qquad (n_1,\ldots,n_d) \in \mathbb{N}^d.
\end{equation}

\item[(2)] There exists $C \geq 1$ such that for any 
polynomial $\phi$ in $\mathcal{P}_d$, 
\begin{equation} \label{calccontr}
\left\| \phi(T_1,\ldots,T_d) \right\| 
\leq C \left\| \phi \right\|_{\infty, \mathbb{D}^d}.
\end{equation}

\item[(3)] There exists a bounded invertible operator
$S\colon H\to H$ such that for any $j=1,\ldots,d$,
$S^{-1} T_j S$ is a contraction.
\end{itemize}
\end{theorem}

\begin{proof}
The proof of (1) will rely on Lemma \ref{lem41}. The Ritt operators 
$T_1,\ldots,T_{d-2}$ are similar to contractions hence
according to \cite[Theorem 8.1]{LM}, $T_k$ has an $H^\infty(B_{\gamma_k})$ functional calculus 
for some $\gamma_k$ in $(0,\frac{\pi}{2})$,
for all $k=1,\ldots,d-2$. The argument in the proof of Theorem \ref{Th42} shows that 
there exist a measure space $\Omega$, unitaries
$V_1,\ldots,V_{d-2}$ on $L^2(\Omega)$ and bounded operators
$$
J_1,\ldots,J_{d-2}\colon H\longrightarrow L^2(\Omega;H)
\qquad\hbox{and}\qquad
Q_1,\ldots,Q_{d-2}\colon L^2(\Omega;H)\longrightarrow H,
$$
such that for any $k=1,\ldots,d-2$, 
$$
T_k^{n_k}=Q_k(V_k\overline{\otimes} I_H)^{n_k}J_k,\qquad n_k\in\Ndb,
$$
and 
$$
J_kR=(I_{L^2(\Omega)}\overline{\otimes}R)J_k
$$
for any $R\colon H\to H$ commuting with $T_k$.
 
By assumption there exists an invertible 
$W\colon H \to H$
such that $W^{-1}T_{d-1}W$ and $W^{-1}T_dW$  are contractions.
By Ando's Theorem \cite{An},
there exist a Hilbert space $L$ containing $H$ as a closed subspace 
and two unitaries $V_{d-1}, V_d\colon L\to L$ such that
\begin{equation} 
(W^{-1} T_{d-1} W)^{n_{d-1}} (W^{-1} T_d W)^{n_d} = 
P_H V_{d-1}^{d-1} V_d^{n_d} J_H, \qquad (n_{d-1},n_d) \in \mathbb{N}^2,
\end{equation}
where $J_H \colon H \to L$ and $P_H=J_H^*\colon L\to H$
denote the inclusion map and the orthogonal projection, respectively.
This can be written as
\begin{equation} \label{Ando}
T_{d-1}^{n_{d-1}} T_d^{n_d} = 
Q_{d-1} V_{d-1}^{d-1} V_d^{n_d} J_{d-1}, \qquad (n_{d-1},n_d) \in \mathbb{N}^2,
\end{equation}
with $Q_{d-1} = W P_H$ and $J_{d-1} = H_H W^{-1}$.

We can therefore apply Lemma \ref{lem41} to $(T_1,\ldots,T_d)$
with $m=d-2$ and $Y=L$. Thus 
there exist two bounded operators 
$J \colon H \to L_2(\Omega^{d-2};L)$ and 
$Q \colon L_2(\Omega^{d-2};L) \to H$, as well as operators $U_1,\ldots,U_d$ 
on $L_2(\Omega^{d-2};L)$ such that
\begin{equation} \label{combcontr}
T_1^{n_1}...T_d^{n_d} = Q U_1^{n_1} \cdots U_d^{n_d} J, 
\qquad (n_1,\ldots,n_d) \in \mathbb{N}^d
\end{equation} 
and the operators $U_k$ are given by
\begin{align*}
& U_k = I^{\otimes k-1} ~\overline{\otimes}~ V_k ~\overline{\otimes} 
I^{\otimes d-2-k} ~\overline{\otimes} ~I_L,  &  k=1,\ldots,d-2; \\
& U_k = I^{\otimes d-2} \overline{\otimes}~ U_k, & k=d-1,d.
\end{align*}
Clearly $K=L_2(\Omega^{d-2};L)$ is a Hilbert space and $U_1,\ldots,U_d$
are commuting unitaries. This shows (1).

(2) is a direct consequence of (1). Indeed for any $\phi\in \mathcal{P}_d$,
(1) implies 
\begin{equation}
\left\| \phi(T_1,\ldots,T_d) \right\| \leq \left\| Q \right\| 
\left\| J \right\| \left\| \phi(U_1,\ldots,U_d) \right\|,
\end{equation}
and by the functional calculus of unitary operators,
\begin{equation} \label{calcunit}
\left\| \phi(U_1,\ldots,U_d) \right\| \leq \norm{\phi}_{\infty,\Ddb^d}.
\end{equation}

We turn now to the proof of (3). 
We appeal to \cite[Proposition 2.4]{AFLM}. 
Consider the algebraic semigroup $\mathcal{G} = 
(\mathbb{N}^d,+)$ and its representations 
\begin{equation}
\pi\colon \begin{array}{llc} \mathcal{G} & \to & B(H) \\ 
(n_1,\ldots,n_d) & \mapsto & T_1^{n_1} \cdots T_d^{n_d} \end{array} 
\quad\hbox{and}\quad \rho \colon \begin{array}{llc} \mathcal{G} & \to & B(K) \\ 
(n_1,\ldots,n_d) & \mapsto & U_1^{n_1} \cdots U_d^{n_d} \end{array},
\end{equation} 
where $K$ and $U_1,\ldots,U_d$ are provided by (1).

According to (\ref{combcontr}), we have two bounded operators 
$J \colon H \to K$ and $Q \colon K \to H$ such that 
\begin{equation}
\pi(n_1,\ldots,n_d) = Q \rho(n_1,\ldots,n_d) J, \qquad (n_1,\ldots,n_d) \in \mathcal{G}.
\end{equation}

Hence by \cite[Proposition 2.4]{AFLM}, there exist two $\rho$-invariant closed 
subspaces $M \subset N \subset K$, as well as 
an isomorphism $S\colon  H \to N/M$ such that the compressed representation 
$\tilde{\rho} \colon  \mathcal{G} \to B(N/M)$ satisfies
\begin{equation} \label{compr}
\pi(n_1,\ldots,n_d) = S^{-1} \tilde{\rho}(n_1,\ldots,n_d) S, 
\qquad (n_1,\ldots,n_d) \in \mathcal{G}.
\end{equation} 
For any $k=1,\ldots,d$, define $R_k \colon  N/M \to N/M$ 
by $R_k(\dot{x}) = \overbrace{U_k(x)}^{\bullet}$ for any $x\in N$,
where $\dot{x}$ denotes its class modulo $M$. Then $R_1,\ldots,R_d$
are contractions and (\ref{compr}) can be equivalenty written as
$$
T_1^{n_1} \cdots T_d^{n_d} = S^{-1} R_1^{n_1} \cdots R_d^{n_d} S,
\qquad (n_1,\ldots,n_d) \in \mathcal{G}.
$$
This implies that
$$
T_k = S^{-1} R_k S
$$
for any $k=1,\ldots, d$. By construction, $N/M$ is a Hilbert space. 
Since it is isomorphic to $H$, through $S$, it is isometrically isomorphic to
$H$. In other words, there exists a unitary $V\colon  N/M\to H$. The above 
identity can be written as
$$
T_k = S^{-1}V^* V R_kV^*VS
$$
for any $k=1,\ldots, d$. 
Now changing $S$ into $VS$ and $R_k$ into $V R_kV^*$, property
(3) follows at once.
\end{proof}

The next corollary is a straighforward consequence of the previous theorem.

Before stating it, we recall that 
Pisier showed in \cite{P4} the existence of a pair $(T_1,T_2)$ of commuting operators
on Hilbert space $H$ such that $T_1$ and $T_2$ are both similar to contractions 
(that is, there exist bounded invertible operators $S_1,S_2\colon  H\to H$
such that $S_1^{-1}T_1S_1$ and $S_2^{-1}T_2S_2$ are contractions) but there
is no  common bounded invertible $S\colon H\to H$ such that 
$S^{-1}T_1S$ and $S^{-1}T_2S$ are contractions.

\begin{corollary}\label{JFC-Hilbert}
Let $d\geq 2$ be an integer and let $(T_1,\ldots,T_d)$ be a 
commuting family of Ritt operators 
on Hilbert space $H$. The following assertions are equivalent.
\begin{itemize}
\item [(1)] $(T_1,\ldots,T_d)$ admits an $H^\infty(B_{\gamma_1}\times
\cdots\times B_{\gamma_d})$ functional calculus 
for some $\gamma_k\in(0,\frac{\pi}{2})$, $k=1,\ldots,d$.
\item [(2)] There exists a bounded invertible operator $S\colon H\to H$ such that
for any $k=1,\ldots,d$, $S^{-1}T_kS$ is a contraction.
\end{itemize}
\end{corollary}

We finally mention that 
Theorem \ref{Th51} and Corollary \ref{JFC-Hilbert}
have semigroup versions, that can be obtained by adapting the
previous arguments. However we omit their statement as they were already 
proved in the paper \cite{LM96} (by using the notion 
of complete boundedness and Paulsen's similarity Theorem).

\section{Appendix: The Franks-McIntosh decomposition on Stolz domains} \label{Appendix}

In this section we provide a detailed proof of the 
Franks-McIntosh decomposition on Stolz domains used in Section \ref{Automatic}. 
As indicated in the Introduction,
this result is implicit in \cite[Section 4]{FMI}, however 
no proof has been written yet. 
The one we provide here is close to the one for sectors given in \cite[Section 3]{FMI}, 
and much simpler that the one which is sketched in \cite[Section 4]{FMI} 
for domains having several points of contact.

\begin{theorem}\label{FranksMcIntoshplusvariables}
Let $d \geq 1$ be an integer, let $\beta_k$ in $( 0, \frac{\pi}{2} )$ 
and $\alpha_k$ in $(0,\beta_k)$, $k=1,\ldots,d$.
There exist sequences $(\Psi_{k,i_k})_{i_k\geq 1}$ and 
$(\tilde{\Psi}_{k,i_k})_{i_k\geq 1}$ in 
$H_0^\infty(B_{\alpha_k})$ verifying the following properties.
\begin{itemize}
\item [(1)] 
For every real number $p>0$ and for any
$k=1,\ldots,d$, 
\begin{equation} \label{majphiij}
\sup\left\lbrace \sum_{i_k = 1}^{\infty} \left| \Psi_{k,i_k}(\zeta_k)\right|^p \, :\, 
\zeta_k \in B_{\alpha_k} \right\rbrace \,<\infty\quad
\hbox{and}\quad \sup\left\lbrace \sum_{i_k = 1}^{\infty} \left| 
\tilde{\Psi}_{k,i_k}(\zeta_k)\right|^p : \zeta_k \in B_{\alpha_k} 
\right\rbrace \,<\infty.
\end{equation}
\item [(2)] 
There exists a constant $C>0$ such that 
for every $h$ in $H^\infty(B_{\beta_1} \times \cdots \times B_{\beta_d})$, 
there exists a 
family $(a_{i_1,\ldots,i_d})_{i_1,\ldots,i_d\geq 1}$ of complex numbers
such that 
\begin{equation} \label{majalphaij}
\left| a_{i_1,\ldots,i_d} \right| \leq C \left\| h 
\right\|_{\infty, B_{\beta_1} \times \cdots \times B_{\beta_d}},
\qquad (i_1,\ldots,i_d)\in \mathbb{N^*}^d,
\end{equation}
and for every $(\zeta_1,\ldots,\zeta_d) $ in $\prod_{k=1}^d B_{\alpha_k}$,
\begin{equation} \label{FMdvarialbes}
h(\zeta_1,\ldots,\zeta_d) = \sum_{i_1,\cdots,i_d\geq 1} a_{i_1,\ldots,i_d} 
\Psi_{1,i_1}(\zeta_1)\tilde{\Psi}_{1,i_1}(\zeta_1) 
\cdots \Psi_{d,i_d}(\zeta_d)\tilde{\Psi}_{d,i_d}(\zeta_d).
\end{equation}
\end{itemize}
\end{theorem}

The main part of the proof will consist in showing the 
following one-variable result.

\begin{proposition} \label{decompFranksMcIntosh}
Let $0<\alpha<\beta<\frac{\pi}{2}$. There exist a sequence 
$(\Phi_i)_{i\geq 1}$ in $H^{\infty}_0(B_\alpha)$ and a constant $C>0$
such that 
$$
\sup\left\lbrace \sum_{i = 1}^{\infty} \left| \Phi_{i}(\zeta)\right|^p \, :\, 
\zeta \in B_{\alpha} \right\rbrace \,<\infty
$$
for any $p>0$, and 
for any $h\in H^\infty(B_\beta)$, there exists a sequence 
$(a_i)_{i\geq 1}$ of complex numbers such that
$\vert a_i\vert\leq C\norm{h}_{\infty,B_\beta}$ for any $i\geq 1$
and
\begin{equation} \label{decomph}
h (\zeta) = \sum_{i=1}^\infty a_{i} \Phi_{i}(\zeta),\qquad \zeta\in B_\alpha.
\end{equation}
\end{proposition}

\begin{remark}\label{InnOut1} 
Since $B_\alpha$ is a simply connected domain bounded by
a rectifiable Jordan curve, any element of $H^\infty(B_\alpha)$ 
admits boundary values. Further for any $\Phi\in H^\infty_0(B_\alpha)$,
there exist $\Psi,\tilde{\Psi}$ in $H^\infty_0(B_\alpha)$ such that
\begin{equation}\label{InnOut2}
\vert \Psi(\zeta)\vert = \vert \tilde{\Psi}(\zeta)\vert =
\vert \Phi(\zeta)\vert^{\frac12},\qquad \zeta\in\partial B_\alpha.
\end{equation}
Indeed given $\Phi\in H^\infty_0(B_\alpha)$, there exists $s>0$ and
$F\in H^\infty(B_\alpha)$ such that $(1-\zeta)^s\Phi(\zeta)=
F(\zeta)$ for any $\zeta\in B_\alpha$. Then using inner-outer factorisation, we may write
$F=\varphi\tilde{\varphi}$ with $\vert\varphi\vert=
\vert\tilde{\varphi}\vert = \vert F\vert^{\frac12}$ on the boundary
of $B_\alpha$. Then we obtain (\ref{InnOut2}) by taking 
$\Psi(\zeta)=(1-\zeta)^{\frac{s}{2}} \varphi(\zeta)$ and
$\tilde{\Psi}(\zeta)=(1-\zeta)^{\frac{s}{2}} \tilde{\varphi}(\zeta)$.

Combining the above factorization property with Proposition 
\ref{decompFranksMcIntosh}, we immediately obtain 
Theorem \ref{FranksMcIntoshplusvariables}
in the case $d=1$.
\end{remark}

Before proceeeding to the proof of Proposition 
\ref{decompFranksMcIntosh}, we need some preliminary constructions.
We fix some $0<\alpha<\mu<\beta<\frac{\pi}{2}$.

We let $\Gamma_0$ 
denote the arc of the circle centered at $0$ with radius $\sin(\mu)$,
joining  $\sin(\mu) e^{i \left(\frac{\pi}{2}-\mu \right)}$
to $\sin(\mu) e^{i \left(\mu-\frac{\pi}{2}\right)}$ counterclockwise. 
Then we let $\Gamma_1$ and  $\Gamma_2$ denote the segments joining 
$1$ to $\sin(\mu) e^{i \left( \frac{\pi}{2} - \mu \right)}$ 
and $\sin(\mu) e^{i \left( \mu - \frac{\pi}{2} \right)}$ to $1$, respectively.
Clearly $\Gamma_0$, $\Gamma_1$ and $\Gamma_2$ divide $\partial B_\mu$.

We divide $\Gamma_0$ into a finite number of arcs 
$\left\lbrace \gamma_{0,k}\right\rbrace_{k=0}^{N} $ with fixed length 
$\delta \leq \frac{1}{2} \text{dist}(\partial B_\alpha, \Gamma_0)$.
For any $0\leq k\leq N$,
we denote by $z_{0,k}$ the center of  $\gamma_{0,k}$
and we let $D_{0,k}$ be the open ball centered at $z_{0,k}$ with radius $\delta$.
Thus $D_{0,k}$ does not intersect $\partial B_\alpha$.

Let $l=\cos(\mu)$; this is the length of the segment $\Gamma_1$. We introduce the
sequence of segments 
\begin{equation*}
\gamma_{1,k} = \left\lbrace z \in \Gamma_1 
\, :\, l \rho^{-k-1} \leq \left| 1-z \right| \leq l \rho^{-k}\right\rbrace, \qquad
k\geq 0,
\end{equation*}
for some $\rho>1$ which will be chosen below. These segments divide $\Gamma_1$.
Let $z_{1,k}$ be the center of $\gamma_{1,k}$ and let $D_{1,k}$ be 
the open ball centered at $z_{1,k}$ with radius 
\begin{equation}\label{Radius}
s_k=l(\rho^{-k} - \rho^{-k-1}).
\end{equation}
We choose $\rho$ such that for every $k \geq 0$, the closure of $D_{1,k}$ 
does not intersect $\partial B_\alpha$.

We divide $\Gamma_2$ in the same manner by setting, for any $k\geq 0$,
\begin{equation*}
\gamma_{2,k} = \left\lbrace \overline{z} : z \in \gamma_{1,k} 
\right\rbrace, \qquad z_{2,k} = \overline{z_{1,k}} 
\qquad\hbox{and}\qquad 
D_{2,k} =  \left\lbrace \overline{z} : z \in D_{1,k} \right\rbrace.
\end{equation*}

For any $\zeta$ in $B_\alpha$ and any $z$ in the union of
$\cup_{k=0}^N D_{0,k}$, $\cup_{k=0}^{\infty} D_{1,k}$ and 
$\cup_{k=0}^{\infty} D_{2,k}$, 
we let
\begin{equation*}
 K(z,\zeta) = \dfrac{(1-z)^\frac{1}{2} (1-\zeta)^\frac{1}{2}}{z-\zeta}\,.
\end{equation*}
For $z,\zeta$ as above, elementary computations yield estimates
\begin{equation} \label{trigo}
|1 - \zeta | \lesssim | z-\zeta | \qquad 
\hbox{and}\qquad |1-z| \lesssim \left| z-\zeta \right|.
\end{equation}
We derive that for $m=1,2$ and for any $r\in\Ndb$, we have estimates
\begin{equation} \label{majKzzeta}
\text{sup} \left\lbrace \left| K(z,\zeta) \right| \,:\, 
z \in D_{m,k},\ \zeta\in B_\alpha,\ l \rho^{-r-1} \leq 
\left| 1- \zeta \right| \leq l \rho^{-r} \right\rbrace 
\lesssim \rho^{-\frac{\left| k-r \right|}{2}}.
\end{equation}
Indeed for $z,\zeta$ as above, we have 
$\left| 1 - z \right| \lesssim \rho^{-k}$,
$\left| 1 - \zeta \right| \lesssim \rho^{-r}$ and 
by (\ref{trigo}), we have $\rho^{-\min(k,r)} \lesssim \left| z-\zeta \right|$. 
These three estimates yield (\ref{majKzzeta}).

It readily follows from the above definitions that 
for $m=1,2$ and $k \geq 0$, we have
\begin{equation}\label{log}
\int_{\gamma_{m,k}} \Bigl\vert \frac{dz}{1-z} \Bigr\vert \,=\, {\rm log}(\rho).
\end{equation}

For $m=1,2$ and $k\geq 0$, we let 
$\left\lbrace e_{m,k,j} \right\rbrace_{j=0}^{\infty}$ be an orthonormal family of 
$L^2\bigl(\gamma_{m,k},\left| \frac{dz}{1-z} \right|\bigr)$ 
such that for any $n\in\Ndb$, ${\rm Span}\{e_{m,k,0},\ldots,e_{m,k,n}\}$
is equal to the subspace of polynomial functions with degree less than or equal to $n$.
Likewise, for $0\leq k\leq N$,
we let $\left\lbrace e_{0,k,j} \right\rbrace_{j=0}^{\infty}$ be an orthonormal family of 
$L^2\bigl(\gamma_{0,k},\left| \frac{dz}{z} \right|\bigr)$ 
such that for any $n\in\Ndb$, ${\rm Span}\{e_{m,k,0},\ldots,e_{m,k,n}\}$
is equal to the subspace of polynomial functions with degree less than or equal to $n$.

Next for any $m\in\{0,1,2\}$ and any $k\geq 0$
(with the convention that $k\leq N$ if $m=0$),
we define $\Phi_{m,k,j}
\colon B_\alpha\to \Cdb$ by
\begin{equation*}
\Phi_{m,k,j} (\zeta) = 
\dfrac{1}{2\pi i} 
\int_{\gamma_{m,k}} \overline{e_{m,k,j}(z)}\, K(z,\zeta) \, \dfrac{dz}{1-z},\qquad \zeta\in B_\alpha.
\end{equation*}
These functions are well defined holomorphic functions
belonging to $H^\infty_0(B_\alpha)$. Indeed according to the definition of $K$
and the Cauchy-Schwarz inequality, we have 
$$
\left| \Phi_{m,k,j}(\zeta) \right|  \leq \,\frac{1}{2\pi}
\left(\int_{\gamma_{m,k}} \left\vert K(z,\zeta)\right\vert^2\, 
\frac{\vert dz\vert}{\vert 1 -z\vert}\right)^{\frac12}
 \leq\, \frac{|1 - \zeta |^\frac{1}{2} }{2\pi}
\left(\int_{\gamma_{m,k}} 
\frac{\vert dz\vert}{\vert \zeta -z\vert^2}\right)^{\frac12}
\lesssim |1 - \zeta |^\frac{1}{2},
$$
since $\vert z-\zeta\vert\geq {\rm dist}(B_\alpha,\gamma_{m,k})>0$.

\begin{lemma} \label{seriephi}
There exists a constant $c>0$ such that if $\zeta \in B_\alpha$ 
satisfies 
\begin{equation}\label{H1}
l\rho^{-r-1} \leq \left| 1- \zeta \right| \leq l\rho^{-r}
\end{equation}
for some $r\in\Ndb$, then for any $k\geq 0$, $j\geq 1$ and $m=1,2$, we have
\begin{equation}\label{H2}
\left| \Phi_{0,k,j} (\zeta) \right| \leq c~~2^{-j} 
\qquad\hbox{and}\qquad 
\left| \Phi_{m,k,j} (\zeta) \right| \leq c~~ 2^{-j} 
\rho^{-\frac{\left| k-r \right|}{2}}.
\end{equation}
\end{lemma}

\begin{proof}
We start proving the second estimate. 
Let $m\in\{1,2\}$ and $k\geq 0$. For any fixed $\zeta\in B_\alpha$, the restriction
of $K(\cdot,\zeta)$ to $D_{m,k}$ is analytic.
Recall (\ref{Radius}) and consider the normalised power series expansion, 
\begin{equation*}
K(z,\zeta) = \sum_{n=0}^{\infty} b_{m,k,n} 
\left( \dfrac{z-z_{m,k}}{s_k} \right)^n.
\end{equation*}
Assume the estimate (\ref{H1}). Then according to (\ref{majKzzeta}), we have
\begin{equation} \label{estKzzeta}
\text{sup} \left\lbrace \left| K(z,\zeta) \right| \,:\, z \in  D_{m,k}
\right\rbrace \lesssim \rho^{-\frac{\left| k-r \right|}{2}}
\end{equation}
Using Bessel-Parseval in $H = L^2\left( \partial D_{m,k}, 
\frac{\left| dz \right| }{2\pi s_k} \right)$ 
and (\ref{estKzzeta}), one obtains
\begin{equation} \label{BesselParseval}
\left( \sum_{n=0}^{\infty} \left| b_{m,k,n} \right|^2 \right)^\frac{1}{2} =  
\left\| K(\cdot,\zeta) \right\|_H = \left( \int_{\partial D_{m,k}} \left| 
K(z,\zeta) \right|^2\frac{|dz|}{2\pi s_k} \right)^\frac{1}{2}  
\lesssim \rho^{-\frac{\left| k-r \right|}{2}}.
\end{equation}

By construction, $\gamma_{m,k}$ is included in the ball centered 
at $z_{m,k}$ with radius $\frac{s_k}{2}$. 
Hence for any  
$z \in \gamma_{m,k}$ and any integer $N \geq 0$, we have 
\begin{align*}
\sum_{n=j}^{\infty} \left|  b_{m,k,n} 
\left( \dfrac{z-z_{m,k}}{s_k} \right)^n \right| &
\leq \left( \sum_{n=j}^{\infty} \left| b_{m,k,n} \right|^2 \right)^\frac{1}{2} 
\left( \sum_{n=j}^{\infty} \left| 
\dfrac{z-z_{m,k}}{s_k} \right|^{2n} \right)^\frac{1}{2}\\
& \leq \left( \sum_{n=j}^{\infty} \left| b_{m,k,n} \right|^2 \right)^\frac{1}{2}
\left( \sum_{n=j}^{\infty} 4^{-n}\right)^{\frac12}\,
\lesssim  \rho^{-\frac{\left| k-r \right|}{2}}\, 2^{-j}.
\end{align*}
Now recall that in $L^2\bigl(\gamma_{m,k},\left| \frac{dz}{1-z} \right|\bigr)$,
$e_{m,k,j}$ is orthogonal to every polynomial function with 
degree $<j$, hence orthogonal 
to $(z-z_{m,k})^n$ for any $n<j$. Further $\frac{dz}{1-z}$ is the opposite 
of $\left| \frac{dz}{1-z} \right|$. This implies that
\begin{align*}
\left| \Phi_{m,k,j}(\zeta) \right|  
& = \left| \dfrac{1}{2\pi i} \int_{\gamma_{m,k}} 
\overline{e_{m,k,j}(z)}\, \sum_{n=j}^{\infty} b_{m,k,n} 
\left( \dfrac{z-z_{m,k}}{s_k} \right)^n \dfrac{dz}{1-z} \right| \\
& \lesssim \rho^{-\frac{\left| k-r \right|}{2}}\, 2^{-j} 
\int_{\gamma_{m,k}} \left| e_{m,k,j}(z)\right| \,\left| 	\dfrac{dz}{1-z} \right| .												
\end{align*}
Applying (\ref{log}), we deduce the second estimate in (\ref{H2}).

The proof of the first estimate is similar, using the 
fact that on each $\gamma_{0,k}$, $\frac{dz}{z}$ is proportional 
to $\left| \frac{dz}{z} \right|$, and 
replacing (\ref{estKzzeta}) by the observation that the set
$$
\biggl\{\frac{z K(z,\zeta)}{1-z}\, :\, z\in \bigcup\limits_{k=0}^{N} D_{0,k},\, \zeta\in B_\alpha\biggr\}
$$
is bounded.
\end{proof}

\begin{proof}[Proof of Proposition \ref{decompFranksMcIntosh}]
Lemma \ref{seriephi} implies that for any $p>0$ and $m\in\{0,1,2\}$,
\begin{equation}\label{summation}
\sup\left\lbrace \sum_{k,j=0}^{\infty} \left| \Phi_{m,k,j}(\zeta)\right|^p \, :\, 
\zeta \in B_{\alpha} \right\rbrace \,<\infty.
\end{equation}

Let $h\in H^\infty(B_\beta)$. By Cauchy's formula,
\begin{equation} \label{Cauchy}
h(\zeta) = \dfrac{1}{2\pi i} \int_{\partial B_\mu} h(z) K(z,\zeta) \,
\dfrac{dz}{1-z}, \qquad \zeta \in B_\alpha.
\end{equation}
For $m=1,2$, $k\geq 0$ and $j\geq 0$, set
\begin{equation}\label{ai}
a_{m,k,j} = \int_{\gamma_{m,k}} h(z) \, e_{m,k,j}(z)\,\left| \dfrac{dz}{1-z} \right|\,.
\end{equation}
Likewise, for $0\leq k\leq N$ and $j\geq 0$, set
\begin{equation}\label{ai2}
a_{0,k,j} = \int_{\gamma_{0,k}} h(z) \, e_{0,k,j}(z)\,\left| \dfrac{dz}{z} \right|\,.
\end{equation}
By the Cauchy-Schwarz inequality and (\ref{log}), we have a uniform estimate
\begin{equation} \label{CSalphamkj}
\left| a_{m,k,j} \right| 
\lesssim\,
\left\| h \right\|_{\infty, B_\beta},
\end{equation}
for $m\in\{0,1,2\}$, $k\geq 0$ and $j\geq 0$.

For $m=1,2$ and $k\geq 0$, let
$H_{m,k}$ denote the subspace of all polynomial functions of 
$L^2(\gamma_{m,k},\left| \frac{dz}{1-z} \right|)$.  
This is a dense subspace. Hence we have a
series expansion 
\begin{equation}\label{hL2}
h_{\vert \gamma_{m,k}} =  \sum_{j=0}^{\infty} a_{m,k,j} \overline{e_{m,k,j}}
\end{equation}
in the latter space.

Likewise, for $0\leq k\leq N$, let $H_{0,k}$ 
denote the subspace of all polynomial functions of 
$L^2(\gamma_{0,k},\left| \frac{dz}{z} \right|)$. 
This is no longer a dense subspace. 
However by Runge's approximation Theorem (see e.g. \cite[Theorem 13.8]{Rud}), every holomorphic 
function on an open neighborhood of $\gamma_{0,k}$ is uniformly 
approximated by polynomials, hence belongs to
$\overline{H_{0,k}}^{\left\| \cdot \right\|_2}$. 
This implies that the series expansion (\ref{hL2}) holds true as well
in this case.

From (\ref{Cauchy}), we can write $h(\zeta)= h_0(\zeta)+h_1(\zeta) + h_2(\zeta)$ for any
$\zeta \in B_\alpha$, where
$$
h_m(\zeta)=\frac{1}{2\pi i} \int_{\Gamma_m} h(z)K(z,\zeta)\, \frac{dz}{1-z}\,.
$$
for each $m=0,1,2$.
The $L^2$-convergence in (\ref{hL2}) a fortiori holds in the $L^1$-sense, hence

\begin{align*}
h_m(\zeta) & =  \frac{1}{2\pi i} \sum_{k=0}^{\infty} 
\int_{\gamma_{m,k}} h(z) K(z,\zeta)\, \dfrac{dz}{1-z} \\
& = \frac{1}{2\pi i} \sum_{k=0}^{\infty}\sum_{j=0}^{\infty} a_{m,k,j}
\int_{\gamma_{m,k}} \overline{e_{m,k,j}(z)}\, 
K(z,\zeta)\, \dfrac{dz}{1-z}\,,
\end{align*}
and hence
\begin{equation}\label{hm}
h_m(\zeta)
= \sum_{k=0}^{\infty}\sum_{j=0}^{\infty} a_{m,k,j} \Phi_{m,k,j}(z).
\end{equation}

After a suitable reindexing, we obtain the result by combining
(\ref{CSalphamkj}), (\ref{hm}) and (\ref{summation}).
\end{proof}

\begin{proof}[Proof of Theorem \ref{FranksMcIntoshplusvariables}]
The case $d=1$ was settled at the end of Remark \ref{InnOut1}.

Assume  that $d=2$. Let $h\in H^\infty(B_{\beta_1}\times B_{\beta_2})$.
Let $(\Phi_{2,i})_{i\geq 1}$ be the sequence of 
$H^{\infty}_{0}(B_{\alpha_2})$ obtained by
applying Proposition \ref{decompFranksMcIntosh} to the couple
$(\alpha_2,\beta_2)$. For any $\zeta_1\in B_{\beta_1}$, the 
one variable function $h(\zeta_1,\cdotp)$
belongs to $H^\infty(B_{\beta_2})$. Hence we have a decomposition
$$
h(\zeta_1,\zeta_2) = \sum_{i=1}^{\infty} a_i(\zeta_1)\Phi_{2,i}(\zeta_2),
\qquad \zeta_1\in B_{\beta_1},\ \zeta_2\in B_{\alpha_2},
$$
with a uniform estimate $\vert a_i(\zeta_1)\vert\leq C_2\norm{h}_{\infty,
B_{\beta_1}\times B_{\beta_2}}$. Recall from the proof of 
Proposition \ref{decompFranksMcIntosh} that 
the complex numbers $a_i(\zeta_1)$ are defined by (\ref{ai}) and (\ref{ai2}). 
This implies that each $a_i\colon B_{\beta_1}\to\Cdb$ is a holomorphic
function. Further 
the above estimates show that for any $i\geq 1$,
$a_i\in H^\infty(B_{\beta_1})$ with $\norm{a_i}_{\infty, B_{\beta_1}}\leq C_2\norm{h}_{\infty,
B_{\beta_1}\times B_{\beta_2}}$.

Let $(\Phi_{1,i})_{i\geq 1}$ be the sequence of 
$H^{\infty}_{0}(B_{\alpha_1})$ obtained by
applying Proposition \ref{decompFranksMcIntosh} to the couple
$(\alpha_1,\beta_1)$. Applying the latter to each $a_i$, we deduce the
existence of a family $(a_{ij})_{i,j\geq 1}$ of complex numbers such that
$$
\vert a_{ij}\vert \leq C_1C_2 \norm{h}_{\infty,
B_{\beta_1}\times B_{\beta_2}},\qquad i,j\geq 1,
$$
for some constant $C_1>0$, and
$$
a_i(\zeta_1)= \sum_{j=1}^{\infty} a_{ij}\Phi_{1,j}(\zeta_1),\qquad
\zeta_1\in B_{\alpha_1},\ i\geq 1.
$$
Since $\sum_j\vert\Phi_{1,j}(\zeta_1)\vert\,<\infty\,$ and 
$\sum_i\vert\Phi_{2,i}(\zeta_2)\vert\,<\infty\,$
for any $(\zeta_1,\zeta_2)\in B_{\alpha_1}\times B_{\alpha_2}$, we deduce
from above that
$$
h(\zeta_1,\zeta_2) = \sum_{i,j=1}^{\infty} a_{ij} \Phi_{1,j}(\zeta_1)\Phi_{2,i}(\zeta_2),\qquad
(\zeta_1,\zeta_2)\in B_{\alpha_1}\times B_{\alpha_2}.
$$
Now using Remark \ref{InnOut1} as in the case $d=1$, we deduce the result in the case $d=2$.

The general case is obtained by iterating this process.
\end{proof}

\bigskip
\noindent
{\bf Acknowledgements.} 
The two authors were supported by the French 
``Investissements d'Avenir" program, 
project ISITE-BFC (contract ANR-15-IDEX-03).

\end{document}